%% file: BDFKMP_FOP2018.tex
\documentclass[12pt]{article}
\usepackage{amsmath}
\usepackage{amssymb}
\usepackage{amsthm}
\usepackage{epsfig}
\usepackage{graphicx}
\usepackage{enumerate}
\usepackage{cite}
\usepackage{psfrag}
\usepackage[latin1]{inputenc}
\usepackage{url}

\newtheorem{theorem}{Theorem}[section]
\newtheorem{conjecture}[theorem]{Conjecture}

\newtheorem{observation}[theorem]{Observation}
\newtheorem{proposition}[theorem]{Proposition}
\theoremstyle{remark}
\newtheorem{remark}[theorem]{Remark}
\newtheorem{example}[theorem]{Example}
\newtheorem{definition}[theorem]{Definition}

\begin{document}

\title{\textbf{Tables, bounds and graphics
 of sizes of complete arcs in the plane $\mathrm{PG}(2,q)$ for all
$q\le321007$ and sporadic $q$ in 
$[323761\ldots430007]$ obtained by an algorithm with fixed
order of points (FOP)}\thanks{The research of D. Bartoli, G. Faina, S. Marcugini, and F. Pambianco was
 supported in part by Ministry for Education, University
and Research of Italy (MIUR) (Project ``Geometrie di Galois e
strutture di incidenza'')
 and by the Italian National Group for Algebraic and Geometric Structures and their Applications
 (GNSAGA - INDAM).
 The research of A.A.~Davydov and A.A.~Kreshchuk was carried out at the IITP RAS at the expense of the Russian
Foundation for Sciences (project 14-50-00150). This work has been carried out using computing resources of the
federal collective usage center Complex for Simulation and Data
Processing for Mega-science Facilities at NRC Kurchatov Institute,
http://ckp.nrcki.ru/.}}
\date{}
\renewcommand{\arraystretch}{0.94}
\author{Daniele Bartoli \\
{\footnotesize Dipartimento di Matematica e Informatica,
Universit\`{a}
degli Studi di Perugia, }\\
{\footnotesize Via Vanvitelli~1, Perugia, 06123, Italy. E-mail:
daniele.bartoli@unipg.it}
 \and Alexander A.
Davydov \\
{\footnotesize Institute for Information Transmission Problems
(Kharkevich
institute), Russian Academy of Sciences}\\
{\footnotesize Bol'shoi Karetnyi per. 19, Moscow,
127051, Russian Federation. E-mail: adav@iitp.ru}
\and Giorgio Faina \\
{\footnotesize Dipartimento di Matematica e Informatica,
Universit\`{a}
degli Studi di Perugia, }\\
{\footnotesize Via Vanvitelli~1, Perugia, 06123, Italy. E-mail:
giorgio.faina@unipg.it} \and Alexey A. Kreshchuk\\
{\footnotesize Institute for Information Transmission Problems
(Kharkevich
institute), Russian Academy of Sciences}\\
{\footnotesize Bol'shoi Karetnyi per. 19, Moscow,
127051, Russian Federation. E-mail: krsch@iitp.ru}
\and Stefano Marcugini and Fernanda Pambianco \\
{\footnotesize Dipartimento di Matematica e Informatica,
Universit\`{a}
degli Studi di Perugia, }\\
{\footnotesize Via Vanvitelli~1, Perugia, 06123, Italy. E-mail:
\{stefano.marcugini,fernanda.pambianco\}@unipg.it} }
\maketitle
\newpage
\noindent\textbf{Abstract.}  In the previous works of the
authors, a step-by-step algorithm FOP which uses any
\emph{fixed order of points} in the projective plane
$\mathrm{PG}(2,q)$ is proposed to construct small complete
arcs. In each step, the algorithm adds to a current arc the
first point in the fixed order  not lying on the bisecants of
the arc. The algorithm is based on the intuitive postulate that
$\mathrm{PG}(2,q)$ contains a sufficient number of relatively
small complete arcs. Also, in the previous papers, it is shown
that the type of order on the points of $\mathrm{PG}(2,q)$ is
not relevant. \textbf{A complete lexiarc} in $\mathrm{PG}(2,q)$ is a complete arc
obtained by the algorithm FOP using the lexicographical order
of points. In this work, we \textbf{collect and analyze the sizes of
complete lexiarcs} in the following regions:
\begin{center}
 \textbf{all} $q\le321007$, $q$ prime power;

 15 sporadic $q$'s in the interval $[323761\ldots430007]$, see \eqref{eq1_L2}.
\end{center}
In the work \cite{BDFKMP-ArXiv2015Greedy}, the \emph{smallest
known} sizes of complete arcs in $\mathrm{PG}(2,q)$ are
collected for all $q\leq160001$, $q$ prime power.
 The sizes of complete arcs, collected in this work and in \cite{BDFKMP-ArXiv2015Greedy}, 
 provide the following upper
bounds on the smallest size $t_{2}(2,q)$ of a complete arc in
the projective plane $\mathrm{PG}(2,q)$:
    \begin{align*}
t_{2}(2,q)&<
0.998\sqrt{3q\ln q}<1.729\sqrt{q\ln q}&\mbox{ for }&&7&\le q\le160001; \\
t_{2}(2,q)&<1.05\sqrt{3q\ln q}<1.819\sqrt{q\ln q}&\mbox{ for }&&7&\le q\le321007.
\end{align*}

Our investigations and results allow to conjecture that the  
 bound\\ $t_{2}(2,q)<1.05\sqrt{3q\ln q}<1.819\sqrt{q\ln q}$\,\, holds for all $q\ge7$.

It is noted
that sizes of the random complete arcs and complete
lexiarcs behave similarly.

This work can be considered as a continuation and development
of the paper \cite{BDFKMP-JGtoappear}.

\textbf{Mathematics Subject Classification (2010).} 51E21,
51E22, 94B05.

\textbf{Keywords.} Projective planes, complete arcs, small
complete arcs, upper bounds, algorithm FOP, randomized greedy
algorithms

\numberwithin{equation}{section}
\section{Introduction. The main results}\label{sec_Intro}

Let $\mathrm{PG}(2,q)$ be the projective plane over the Galois
field $\mathbb{F}_{q}$ of $q$ elements. An $n$-arc is a set of
$n$ points no three of which are collinear. An $n$-arc is
called complete if it is not contained in an $(n+1)$-arc of
$\mathrm{PG}(2,q)$. For an introduction to projective
geometries over finite fields see \cite
{HirsBook,HirsSt-old,HirsStor-2001,SegreIntrodGalGeom}.

 The relationship among the
theory of $n$-arcs, coding theory, and mathematical statistics
is presented in \cite{EtzStorm2016,HirsSt-old,HirsStor-2001,Klein-Stor,Landjev,LandjevStorme2012}. In particular, a complete arc
in the plane $\mathrm{PG}(2,q),$ the points of which are
treated as 3-dimensional $q$-ary columns, defines a parity
check matrix of a $q$-ary linear code with codimension 3,
Hamming distance 4, and covering radius 2. Arcs can be
interpreted as linear maximum distance separable (MDS) codes
\cite[Sec.~7]{szoT93}, \cite{thaJ92d} and they are related to
optimal coverings arrays \cite{Hartman-Haskin}, superregular
matrices \cite {Keri}, and quantum codes \cite{BMP-quantum}.

A point set $S\subset \mathrm{PG}(2,q)$ is 1-\emph{saturating}
if any point of $ \mathrm{PG}(2,q)\setminus S$ is collinear
with two points in $S$ \cite
{BDGMP_SatSetArxiv,BDGMP_SatSet_Petersb,BDGMP_5ICMCTA,BFMP-JG2013,DGMP-AMC,DMP-JCTA2003,Giul2013Survey,Nagy,Ughi-sat}. In the literature, 1-saturating sets are also called
``saturated sets'', ``spanning sets'', and ``dense sets''.
The points of a 1-saturating set in $\mathrm{PG}(2,q)$ form a
parity check matrix of a \emph{linear covering code} with
codimension 3,  Hamming distance 3 or 4, and covering radius 2.
An open problem is to find small 1-saturating sets
(respectively, short covering codes). A complete arc in
$\mathrm{PG}(2,q)$ is, in particular, a 1-saturating set; often
the smallest known complete arc is the smallest known
1-saturating set \cite
{BFMP-JG2013,DGMP-AMC,DMP-JCTA2003,Giul2013Survey,Pace-A5}. Let
$\ell_{1}(2,q)$ be the smallest size of a 1-saturating set in
$\mathrm{PG}(2,q)$. In \cite {BDGMP_5ICMCTA}, combining approaches of \cite{Nagy} and \cite{BDFKMP-PIT2014}, the following
upper bound  is proved for arbitrary
$q$:
\begin{equation}  \label{eq1_1sat}
\ell_{1}(2,q)\le\sqrt{q(3\ln q+\ln\ln q)}+\sqrt{\frac{q}{3\ln q}}+3.
\end{equation}

Let $t_{2}(2,q)$ be \textbf{the smallest size of a complete arc
in the projective plane $\mathrm{PG}(2,q)$.}

One of the main problems in the study of projective planes,
which is also of interest in coding theory, is the determination of
the spectrum of possible sizes of complete arcs. In particular,
the value of $t_{2}(2,q)$ is interesting. Finding estimates of
the minimum size $t_{2}(2,q)$ is a hard open problem, see e.g.
\cite[Sec.\,4.10]{HirsThas-2015}.

This paper is devoted to \emph{upper bounds} on $t_{2}(2,q)$.

Surveys of results on the sizes of plane complete arcs, methods
of their construction, and the comprehension of the relating
properties can be found in\cite{BDFKMP-ArXiv2013,BDFKMP-ArXiv,%
BDFKMP-PIT2014,BDFKMP-ArXiv2015Greedy,BDFKMP-ArXiv2015FOP,BDFKMP-JGtoappear,BDFMP-DM,BDFMP-JG2013,BDFMP-JG2015,BDMP-3cycle,%
DFMP-JG2005,DFMP-JG2009,FP,Giul2013Survey,%
HirsSurvey83,HirsBook,HirsSt-old,HirsStor-2001,HirsThas-2015,KV,pelG77,pel93,%
SegreIntrodGalGeom,SZ,szoT87a,szoT89survey,szoT93,Szonyi1997surveyCurves}.

Problems connected with small complete arcs in $\mathrm{PG}(2,q)$ are considered in \cite%
{abaV83,AliPhD,Ball-SmallArcs,BDFKMP-Bulg2013,BDFKMP-ArXiv2013,BDFKMP-ArXiv,BDFKMP-ACCT2014Conject,%
BDFKMP-JGtoappear,BDFMP-DM,BDFMP-JG2013,ComputBound-Svetlog2014,BDFKMP-PIT2014,%
BDFKMP-ArXiv2015Greedy,BDFKMP-ArXiv2015FOP,BDFMP-Bulg2012a,BDFMP-JG2015,BDMP-Bulg2012b,BFMP-JG2013,%
BFMPD-ENDM2013,BDFMP-ArXivFOP,BDFMP-ArXivRandom,Blokhuis,CoolStic2009,CoolStic2011,%
DFMP-JG2005,DFMP-JG2009,DGMP-Innov,DGMP-JCD,%
DGMP-AMC,DMP-JCTA2003,FainaGiul,FMMP-1977,FP,FPDM,GacsSzonyi,Giul2000,%
Giul2007affin,Giul2007even,Giul2013Survey,GKMP-A6invar,GiulUghi,Gordon,Hadnagy,HirsSurvey83,%
HirsBook,HirsSad,HirsSt-old,HirsStor-2001,HirsThas-2015,KV,korG83a,%
LisPhD,LisMarcPamb2008,MMP-q29,MMP-q25,Ost,Pace-A5,pelG77,%
pel93,Polv,SegreIntrodGalGeom,SZ,szoT87a,szoT87b,%
szoT89survey,szoT93,Szonyi1997surveyCurves,Ughi-sqrt-log2,UghiAlmost,Voloch87,Voloch90}%
, see also the references therein.

The exact values of $t_{2}(2,q)$ are known only for $q\leq 32$;
see \cite
{AliPhD,BFMP-JG2013,CoolStic2009,CoolStic2011,FMMP-1977,Gordon,HirsBook,HirsSad,MMP-q29,MMP-q25}.

The following lower bounds hold (see \cite
{Ball-SmallArcs,Blokhuis,Polv} and the references
therein):
\begin{equation}
t_{2}(2,q)>\left\{
\begin{array}{ll}
\sqrt{2q}+1 & \text{for any }q, \\
\sqrt{3q}+\frac{1}{2} & \text{for }q=p^{h},\text{ }p\text{ prime, }h=1,2,3.
\end{array}
\right.   \label{eq1_trivlower}
\end{equation}

Let $t(\mathcal{P}_{q})$ be the size of the smallest complete
arc in any (not necessarily Desarguesian) projective plane
$\mathcal{P}_{q}$ of order $q$. In \cite{KV}, for $q$
\emph{large enough}, the following result is proved by
\emph{probabilistic methods} :
\begin{equation}
t(\mathcal{P}_{q})\leq \sqrt{q}\ln^{C}q,\text{ }C\leq 300,
\label{eq1_KimVu_c=300}
\end{equation}
where $C$ is a constant independent of $q$ (so-called universal
or absolute constant). Surveys and results of random constructions for geometrical objects in $\mathrm{PG}(2,q)$ can be
found in \cite {BDGMP_SatSetArxiv,BDGMP_SatSet_Petersb,BDGMP_5ICMCTA,BorSzonTicDefinSets,GacsSzonyi,KV,Kovacs,Nagy}.

In $\mathrm{PG}(2,q)$, complete arcs are obtained by
\emph{algebraic constructions} (see
\cite[p.\thinspace209]{HirsStor-2001}) with sizes approximately
$\frac{1}{3}q$ \cite
{abaV83,BDFMP-DM,korG83a,SZ,szoT87a,szoT87b,Voloch90},
$\frac{1}{4}q$ \cite {BDFMP-DM,korG83a,szoT89survey},
$2q^{0.9}$ for $q>7^{10}$ \cite{SZ}, and $ 2.46q^{0.75} \ln q$
for big prime $q$ \cite{Hadnagy}. It is noted in \cite[
Sec.~8]{GacsSzonyi} that the smallest size of a complete arc in
$ \mathrm{PG}(2,q)$ obtained via algebraic constructions is
\begin{equation}\label{eq1_cq34}
cq^{3/4}
\end{equation}
 where $c$ is a universal constant  \cite[ Sec.\ 3]{szoT89survey}, \cite[ Th.\
6.8]{szoT93}.

There is a substantial gap between the known upper bounds
and the lower bounds on $t_{2}(2,q)$, see
\eqref{eq1_trivlower}--\eqref{eq1_cq34}. The gap is 
reduced if one considers the lower bound~\eqref{eq1_trivlower}
for complete arcs and the upper bound \eqref{eq1_1sat} for
1-saturating sets. However, though complete arcs are
1-saturating sets, they represent a narrower class of objects.
Therefore, for complete arcs, one may not use the bound
\eqref{eq1_1sat} directly. Nevertheless, the common nature of
complete arcs and 1-saturating sets allows to hope for upper
bounds on $t_{2}(2,q)$ similar to~\eqref{eq1_1sat}. The hope is
supported by numerous experimental data and some probabilistic conjectures, see below.

 In  \cite[p.\ 313]{KV}, it is noted (with reference to the work \cite{Blokhuis}) that in a
preliminary report of 1989, J.C. Fisher  obtained by computer
search complete arcs in many planes of small orders and
conjectured that the average size of a complete arc in
$\mathrm{PG}(2,q)$ is about $\sqrt{ 3q\log q}$.

In \cite{BDFKMP-PIT2014}, see also
\cite{BDFKMP-ACCT2014Conject}, an attempt to obtain a
theoretical upper bound on $t_{2}(2,q)$ with the main term of
the form $c\sqrt{q\ln q}$, where $c$ is a small universal
constant, is done. The reasonings of \cite{BDFKMP-PIT2014} are
based on the explanation of the working mechanism of a step-by-step
greedy algorithm for constructing complete arcs in
$\mathrm{PG}(2,q)$ and on quantitative estimations of the
algorithm. For more than half of the steps of the iterative
process, these estimations are proved rigorously. The natural
(and well-founded) conjecture that they hold for the rest of
steps is done, see \cite[Conject.\ 2]{BDFKMP-PIT2014}. As a
result, in \cite{BDFKMP-PIT2014},  the following conjectural upper bounds are formulated.

\begin{conjecture}\emph{\textbf{\cite{BDFKMP-PIT2014}}}\label{conj1}
Let $t_{2}(2,q)$ be the smallest size of a complete arc in the
projective plane $\mathrm{PG}(2,q)$. Under conjecture given in
\emph{\cite[Conject.\ 2]{BDFKMP-PIT2014}}, the following
upper bounds hold:
    \begin{align}
&t_{2}(2,q)<\sqrt{q(3\ln q+\ln\ln q+\ln 3)}+\sqrt{\frac{q}{3\ln q}}+3,
 \label{eq1_Prob bounds}\displaybreak[3]\\
  &t_{2}(2,q)<1.87\sqrt{q\ln q}<1.08\sqrt{3q\ln q}.\label{eq1_!<}
  \end{align}
\end{conjecture}
Moreover, in \cite{BDFKMP-PIT2014} it is conjectured that the
upper bounds \eqref{eq1_Prob bounds}, \eqref{eq1_!<} hold for
all $q$ without any extra conditions.

 Denote by $\overline{t}_{2}(2,q)$ \textbf{the smallest
\emph{known }size of a complete arc in the projective plane
$\mathrm{PG}(2,q)$.} Clearly,
\begin{equation*}
t_{2}(2,q)\le\overline{t}_{2}(2,q).
\end{equation*}

For even $q=2^{h}$, small complete arcs in planes are a base
for inductive infinite families of small complete caps in the
projective spaces $\mathrm{PG}(N,q)$, see \cite{DGMP-JCD}. For
 $h\leq 17$, the smallest known sizes of complete arcs in
$\mathrm{PG}(2,2^{h})$ are collected in \cite{BDFKMP-ArXiv,BDFKMP-PIT2014},
see also
\cite{BDFMP-DM,DFMP-JG2009,DGMP-JCD,BDFMP-JG2013,FainaGiul}. Also, 
${(6\sqrt{q}-6)}$-arcs in $\mathrm{PG}(2,4^{2a+1})$, are
constructed in \cite{DGMP-Innov}; for $a\leq 4$ it is proved
that they are complete. This gives a complete $3066$-arc in
$\mathrm{PG}(2,2^{18})$. In particular, it holds that
\begin{align*}
&\overline{t}_{2}(2,2^{9})=85,~ \overline{t}_{2}(2,2^{10})=124,~\overline{t}_{2}(2,2^{11})=199,~\overline{t}_{2}(2,2^{12})=300,~
\overline{t}_{2}(2,2^{13})=449,\displaybreak[3]\notag\\
&\overline{t} _{2}(2,2^{14})=665,~\overline{t}_{2}(2,2^{15})=987,~\overline{t}_{2}(2,2^{16})=1453,~
\overline{t}_{2}(2,2^{17})=2141,~\overline{t}_{2}(2,2^{18})=3066.
\end{align*}

For prime powers $q\leq 13627$, the values of
$\overline{t} _{2}(2,q)$ (up to January 2013) are collected in
 \cite {BDFMP-DM,BDFMP-JG2013}. For prime powers $q\leq 49727$ and prime $q\le 150001$, the values of
$\overline{t} _{2}(2,q)$ (up to August 2014) are collected in
 \cite {BDFKMP-ArXiv}. For $q\leq 151,$ a number of improvements of
$\overline{t}_{2}(2,q)$, in comparison with
\cite{BDFMP-DM,BDFMP-JG2013}, are given in \cite{Pace-A5} and
cited in \cite{BDFKMP-ArXiv}.

From the results of \cite{BDFKMP-ArXiv,BDFMP-DM,BDFMP-JG2013}, see also \cite{DGMP-JCD,FainaGiul,Giul2000,GiulUghi}, it follows that
\begin{align*}
&t_{2}(2,q)<4\sqrt{q}\quad \text{for }q\leq 841,\text{ }
q=857,31^{2},2^{10},37^{2},41^{2},7^{4};\\
& t_{2}(2,q)<4.5\sqrt{q} \text{ for }q\leq 2647,\text{ }
q=2659,2663,2683,2693,2753,2801; \displaybreak[3] \\
& t_{2}(2,q)<5\sqrt{q} \text{ for }q\leq 9497,\text{ }
q=9539,9587,9613,9623,9649,9689,9923,9973;\\
&t_{2}(2,q)<5.5\sqrt{q} \text{ for }q\leq 38557,\text{ }
q\neq 36481,37537,37963,38039,38197.
\end{align*}

Let $Q_1$ be a set of 34 sporadic $q$'s in the interval $[160801\ldots430007]$, see \cite[Tab.\,3]{BDFKMP-ArXiv2015Greedy}. Let  $Q$ be a set of values of $q$.  We have
\begin{align}
&Q_1=\{160801,161009,162007,163003,164011,165001,166013,167009,168013,\label{eq1_Q1}\\
&169003,170503,178169,180001,185021,190027,200003,210011,250007,260003,\notag\\
&262144,270001,280001,290011,300007,330017,350003,360007,370003,380041,\notag\\
&390001,400009,410009,420001,430007\};\notag\\
&Q=\{2\le q\le160001, q \mbox{ prime power}\}\,\cup Q_1.\label{eq1_Q}
\end{align}

For $q\in Q$, the values of $\overline{t}_{2}(2,q)$ are obtained in the works
\cite{BDFKMP-Bulg2013,BDFKMP-ArXiv2013,BDFKMP-ArXiv,BDFKMP-PIT2014,BDFKMP-ArXiv2015Greedy,BDFKMP-JGtoappear,BDFMP-DM,BDFMP-JG2013,BDMP-Bulg2012b}, see also the references therein.

\textbf{All the
 values of $\overline{t}_{2}(2,q)$ (up to June
2015),  $q\in Q$, are collected in the work \cite{BDFKMP-ArXiv2015Greedy}.}

Note that, for $q\in Q$, in the works
\cite{BDFKMP-Bulg2013,BDFKMP-ArXiv2013,BDFKMP-ArXiv,BDFKMP-PIT2014,BDFKMP-ArXiv2015Greedy,%
BDFKMP-JGtoappear,BDFMP-DM,BDFMP-JG2013,ComputBound-Svetlog2014},
the most of the values $\overline{t }_{2}(2,q)$ have been
obtained by computer search using \emph{randomized greedy algorithms}
described in \cite
{BDFKMP-JGtoappear,BDFMP-DM,BDFMP-JG2013,BDFMP-JG2015,DFMP-JG2005,DFMP-JG2009,DMP-JG2004}.
In each step, a step-by-step greedy algorithm adds to
    an incomplete current arc  a point providing the maximum
    possible (for the given step) number of new covered
    points.  

We denote by $t_{2}^{G}(2,q)$ \textbf{the size of a complete arc  in
 $\mathrm{PG}(2,q)$ collected in \cite{BDFKMP-ArXiv2015Greedy}.}

An important way to obtain small plane complete arcs is
a step-by-step \emph{algorithm with fixed order of points
(FOP)}, see
\cite{BDFKMP-ArXiv2015FOP,BDFKMP-JGtoappear,ComputBound-Svetlog2014,BDFMP-JG2015,BDMP-Bulg2012b,BFMPD-ENDM2013,BDFMP-ArXivFOP}.
The algorithm FOP fixes a particular order on points of
$\mathrm{PG}(2,q)$. In each step, the algorithm FOP adds to an
incomplete current arc the next point in this order  not lying
on bisecants of this arc. For both prime and non-prime $q$, a
lexicographical order of points can be used, see
Section~\ref{sec-FOP}.

\begin{definition}
\textbf{We call a \emph{lexiarc} an arc obtained by the
algorithm FOP using the lexicographical order of points}.
\end{definition}

 Let $t_{2}^{L}(2,q)$ be \textbf{the size of a complete lexiarc  in
the projective plane $\mathrm{PG}(2,q)$.}

Note that the sizes of complete arcs obtained by the algorithm
FOP vary insignificantly with respect to the order of points,
see \cite{BDFKMP-ArXiv2015FOP,BDMP-Bulg2012b,BFMPD-ENDM2013,BDFMP-JG2015}.
In particular, for the lexicographical order of points described in Section \ref{sec-FOP}, in the case of  \emph{prime $q$,  the size  $t_{2}^{L}(2,q)$ of a complete
lexiarc and its set of points depend on $q$ only}. No other
factors affect size and structure of a complete lexiarc.

Let $L_{1}$, $L_{2}$, $L_{3}$,  and $L$ be the
following sets of values of $q$:
\begin{align}
L_{1}&=\{q\le  301813,~q\mbox{ prime power}\};\label{eq1_L1}\displaybreak[3]\\
L_2&=\{323761,326041,330017,332929,340007,344569,350003,360007,\label{eq1_L2}\\
&\qquad 370003,380041,390001,400009,410009,420001,430007\};\displaybreak[3]\notag\\
L_3&=\{301813< q\le321007, ~q\mbox{ prime power}\};\label{eq1_L3}\displaybreak[3]\\
L&=L_1\cup L_2\cup L_3=\{q\le321007,~ q \mbox{ prime power}\}\,\cup L_2.\label{eq1_L}
\end{align}
For $q\in L_{1}\cup L_2$, the
 values of $t_{2}^{L}(2,q)$  are obtained in
 \cite{BDFMP-ArXivFOP,ComputBound-Svetlog2014,BDFKMP-ArXiv2015FOP,BDFKMP-JGtoappear,BDFMP-JG2015,BFMPD-ENDM2013}
 and collected in the work \cite{BDFKMP-ArXiv2015FOP}. For $q\in L_{3}$, the
 values of $t_{2}^{L}(2,q)$  \emph{are obtained in  this work}.

\textbf{All the
 values of $t_{2}^{L}(2,q)$,  $q\in L$, are collected in this work.}

It should be emphasized that in this work, to obtain upper bounds, we use the
computer search results for \textbf{all prime powers} $q$ in the
regions $q\le 160001$ with greedy algorithms and $q\le321007$
with the algorithm FOP. In this sense, we say that
\textbf{computer search} in the noted regions is \textbf{complete}.

Complete arcs obtained by greedy algorithms have smaller sizes
than complete lexiarcs, however the greedy algorithms take
essentially greater computer time than the algorithm FOP. This
is why the complete computer search for all  prime powers $q$
with the help of greedy algorithms is done for $q\le160001$
\cite{BDFKMP-ArXiv2015Greedy} whereas the complete search by
algorithm FOP is executed for $q\le321007$. So, we have, see Figure \ref{fig_1} in Section \ref{sec_bounds},
\begin{align}\label{eq1_the best}
\overline{t}_2(2,q)=\left\{
\begin{array}{ccl}
t_{2}^{G}(2,q)&\text{if}&q\in Q=\{2\le q \le160001\} \cup  Q_1\smallskip\\
t_{2}^{L}(2,q)&\text{if}&q\in L\setminus Q=(\{160001< q \le321007\} \cup  L_2)\setminus Q_1
\end{array}
\right..
\end{align}

 In the works of the authors
\cite{BDFKMP-Bulg2013,BDFMP-JG2013,BDFMP-JG2015,BFMPD-ENDM2013,BDMP-Bulg2012b},
non-standard types of upper bounds on $t_{2}(2,q)$ are
proposed. In particular, in the paper \cite{BDFKMP-JGtoappear}, the function
$h(q)$ is defined so that
\begin{align}
&t_{2}(2,q)=h(q)\sqrt{3q \ln q}.\label{eq1_h(q)}
\end{align}
Using the
function $h(q)$ instead of
$t_{2}(2,q)$ allows to do estimates and graphics more
expressive, see Figures \ref{fig_1}, \ref{fig_6} in Section \ref{sec_bounds}.

The following theorem summarizes the main results of this work. The theorem is based
on the complete computer search, the results of which are
collected in
\cite{BDFKMP-ArXiv2013,BDFKMP-ArXiv,BDFKMP-ArXiv2015Greedy,BDFKMP-ArXiv2015FOP,BDFKMP-JGtoappear,BDFMP-DM,BDFMP-JG2013,BDFMP-ArXivFOP},
(see also the references therein) and in Tables 1 -- 6 of this paper.
\begin{theorem}\label{th1_main}
Let $t_{2}(2,q)$ be the smallest size of a complete arc in the
projective plane $\mathrm{PG}(2,q)$. Let  $\overline{t}_{2}(2,q)$ be the smallest
known size of a complete arc in $\mathrm{PG}(2,q)$. Let $L_2$ be given by \eqref{eq1_L2}.
 The following upper bounds hold:
    \begin{align}
&t_{2}(2,q)<\overline{t}_{2}(2,q)<
0.998\sqrt{3q\ln q}<1.729\sqrt{q\ln q}~~\mbox{ for }~~7\le q\le160001;\label{eq1_Bounds1G}\displaybreak[0]\\
&t_{2}(2,q)<\overline{t}_{2}(2,q)<1.05 \sqrt{3q\ln q}<1.819\sqrt{q\ln q}~~\mbox{ for }~~7\le q\le321007\mbox{ and }q\in L_2.
 \label{eq1_Bounds1L}
\end{align}

For $q\le160001$, complete arcs in
$\mathrm{PG}(2,q)$ satisfying the upper bounds
\eqref{eq1_Bounds1G} can be constructed with the
    help of the step-by-step greedy algorithm which adds to
    an incomplete current arc, in each step, a point providing the
maximum
    possible (for the given step) number of new covered
    points.

For $160001<q\le321007$ and $q\in L_2$, complete arcs in
$\mathrm{PG}(2,q)$ satisfying the upper bounds
\eqref{eq1_Bounds1L} can be constructed as lexiarcs with the
    help of the algorithm FOP using the lexicographical order of
    points. The only exception is $q=178169$ for which a complete arc satisfying
    \eqref{eq1_Bounds1L} can be obtained by the greedy algorithm.
\end{theorem}

Calculations executed for sporadic $q\le430007$ ($q\in  L_2$ of  \eqref{eq1_L2}) strengthen the
confidence for validity of the bounds of Theorem
\ref{th1_main}, see also Figures \ref{fig_1} --   \ref{fig_3} in Section \ref{sec_bounds}. Also, it is important that the bounds
\eqref{eq1_Bounds1G}, \eqref{eq1_Bounds1L} are close to the
conjectural (but well-founded) bounds of \cite{BDFKMP-PIT2014},
see \eqref{eq1_Prob bounds}, \eqref{eq1_!<} in Conjecture~\ref{conj1} and Figures \ref{fig_4}, \ref{fig_5} in Section  \ref{sec_bounds}. On the whole, our investigations and results (again see figures in Section  \ref{sec_bounds}) allow to conjecture that the bound \eqref{eq1_Bounds1L} holds for all $q$.

\begin{conjecture}\label{conj1_for all q} 
Let $t_{2}(2,q)$ be the smallest size of a complete arc in the
projective plane $\mathrm{PG}(2,q)$. 
The following upper bound holds:
   $$
t_{2}(2,q)<1.05 \sqrt{3q\ln q}<1.819\sqrt{q\ln q}\text{ for all }q\ge7.
$$
\end{conjecture}

The paper is organized as follows. In Section \ref{sec-FOP}, the algorithm with fixed order of points (FOP) is considered. The lexicographical order of points is described and lexiarcs are defined.
In Section \ref{sec_bounds}, upper bounds on $t_{2}(2,q)$, based on
sizes  of complete lexiarcs in
$\mathrm{PG}(2,q)$ and on sizes of complete arcs collected in \cite{BDFKMP-ArXiv2015Greedy}, are proposed.  In Section \ref{sec-random}, the common nature lexiarcs and random arcs is considered. In Section \ref{sec_list tables},
the list of tables with sizes of complete lexiarcs in the projective plane
 $\mathrm{PG}(2,q)$ is given. In Conclusion the results of this paper are briefly analyzed. In Appendix,  tables of sizes of complete lexiarcs in the
projective plane PG(2,q), obtained in this paper and in the previous works of the authors, are collected.

\section{Algorithm with fixed order of points (FOP). Lexiarcs}
\label{sec-FOP}
\subsection{Step-by-step algorithm FOP}

We use results and approaches of the works
\cite{BDFKMP-ArXiv2015FOP,ComputBound-Svetlog2014,BDFMP-JG2015,BDMP-Bulg2012b,BFMPD-ENDM2013,BDFMP-ArXivFOP}.
Consider the projective plane $\mathrm{PG}(2,q)$ and fix a
particular order on its points. The algorithm FOP builds a
complete arc \textbf{iteratively, step-by-step}.

Let $K^{(j)}$ be the arc obtained on the $j$-th step. On the
next step, the first point in the fixed order  not lying on the
bisecants of $K^{(j)}$ is added to $K^{(j)}$.

\textbf{Algorithm FOP.} Suppose that the points of
$\mathrm{PG}(2,q)$ are ordered as
$A_{1},A_{2},\ldots,A_{q^2+q+1}$.  Consider the empty set as
root of the search and let $K^{(j)}$ be the partial solution
obtained in the $j$-th step, as extension of the root. We put
\begin{align}
&K^{(0)}=\emptyset, \, K^{(1)}=\{A_{1}\},\,K^{(2)}=\{A_{1},A_{2}\},\, m(1)=2,
\displaybreak[0]\label{eq3_FOPalgorithm}\\
& K^{(j+1)}=K^{(j)}\cup\{A_{m(j)}\},~ \notag\displaybreak[0]\\
&m(j)=\min \{ i \in [m(j-1)+1,q^2+q+1]\: |\: \nexists \;P,Q \in K^{(j)}  : \, A_{i},P,Q \textrm{ are collinear}\},\notag
\end{align}
i.e. $m(j)$ is the minimum subscript $i$ such that the
corresponding point $A_{i}$ is not saturated by $K^{(j)}$. The
process ends when a complete arc is obtained.
\begin{remark}
\label{rem3_lexicodes}
A point of $\mathrm{PG}(2,q)$ can be treated as a 3-vector
column of a code parity check matrix. Then, formally, the algorithm FOP can be considered as a version of the recursive g-parity check algorithm for greedy codes of codimension $r=3$ and minimum distance $d=4$, see e.g. \cite[p. 25]{BrPlessGreedy}, \cite{MonroePlessGreedCod}, \cite[Section 7]{PlessHandb}.

However,  in
 coding theory, for given
$r,d$, the aim is to get a \emph{long} code while our goal is
to obtain a \emph{short} complete arc. Moreover, our estimates
and computer search show that for $r=3,$ $d=4$, the algorithm
FOP gives \textquotedblleft bad\textquotedblright\ codes,
essentially shorter than \textquotedblleft
good\textquotedblright\ codes corresponding to ovals and
hyperovals.

Finally, note that we do not use a word
\textquotedblleft greedy\textquotedblright in the name of the
algorithm FOP, as in projective geometry the terms
\textquotedblleft greedy algorithm\textquotedblright\ and
\textquotedblleft randomized greedy
algorithm\textquotedblright\ are traditionally connected with
other approaches, see \cite
{BDFKMP-JGtoappear,BDFMP-DM,BDFMP-JG2013,BDFMP-JG2015,DFMP-JG2005,DFMP-JG2009,DMP-JG2004}.
\end{remark}

\subsection{Lexicographical order of points}
In the beginning, we consider $q$  prime. Let the elements of
the field $\mathbb{F}_{q}=\{0,1,\ldots ,q-1\}$ be treated as
integers modulo $q$. Let the points $A_{i}$ of
$\mathrm{PG}(2,q)$ be represented in homogenous coordinates so
that
\begin{align}
\label{eq3_homo}
A_{i}=(x_{0}^{(i)},x_{1}^{(i)},x_{2}^{(i)}),~ x_{j}^{(i)}\in
\mathbb{F}_{q},
\end{align} where the leftmost non-zero element is 1.
For $A_i$, we put
\begin{align}\label{eq3_lexi}
i=x_{0}^{(i)}q^{2}+x_{1}^{(i)}q+x_{2}^{(i)}.
\end{align}
So, the homogenous coordinates of a point $A_{i}$ are treated
as its number $i$ written in the $q$-ary scale of notation.
Recall that the points of $\mathrm{PG}(2,q)$ are ordered as
$A_{1},A_{2},\ldots,A_{q^2+q+1}$.

It is important that \textbf{for such lexicographical order for
\emph{prime} $q$, the size $t^{L}_{2}(2,q)$ of a complete
lexiarc and its set of points depend on $q$ only}. No other
factors affect size and structure of a complete lexiarc.

Now let $q=p^{m}$, $p$ prime, $m\ge2$. Let $F_{q}(x)$ be a
primitive polynomial of $\mathbb{F}_{p^{m}}$ and let $\alpha$
be a root of $F_{q}(x)$. Elements of $\mathbb{F}_{p^{m}}$ are
represented by integers as follows
$$
\mathbb{F}_{p^{m}}=\mathbb{F}_{q}=\{0,1=\alpha^{0},2=\alpha^{1},\ldots,q-1=\alpha^{q-2}\}.
$$
Then we again use \eqref{eq3_homo} and \eqref{eq3_lexi}.
However for non-prime $q$ the size $t^{L}_{2}(2,q)$ of a
complete lexiarc depends on $q$ and on the form of the
polynomial $F_{q}(x)$. In this work we use primitive
polynomials that are created by the program system MAGMA \cite{MAGMA} by
default, see Table~A where polynomials for $q<36000$ are given.
In any case, the choice of the polynomial changes the size of
complete lexiarc inessentially.

We have noted in Introduction that, in general, the sizes of
complete arcs obtained by the algorithm FOP vary
insignificantly with respect to the order of points. In
particular, in
\cite{BDMP-Bulg2012b,BFMPD-ENDM2013,BDFMP-JG2015} the so-called
Singer order of points (based on the Singer group of
collineations) is considered and it is shown that in the region
$5000<q\le40009$, $q$ prime, the difference between the sizes
of complete arcs obtained using the lexicographical order and
the Singer order is less then $2\%$ and none of the two orders
gives the smallest size for all $q$.

\begin{table*}[htbp]
\textbf{Table A.} Primitive polynomials used for lexiarcs in $\mathrm{PG}(2,q)$ with non-prime $q$\medskip\\
\small{
\renewcommand{\arraystretch}{1.0}
\begin{tabular}{@{}c@{\,}|@{\,}c@{\,}||@{\,}c@{\,}|@{}c@{\,}||@{\,}c@{\,}|@{\,}c@{}}\hline
 $q=p^{m}$&primitive&$q=p^{m}$&primitive&$q=p^{m}$& primitive \\
&polynomial&&polynomial&&polynomial\\ \hline
 $4=2^{2}$& $x^{2}+x+1$&$8=2^{3}$& $x^{3}+x+1$&$9=3^{2}$& $x^{2}+2x+2$\\
$16=2^{4}$& $x^{4}+x^{3}+1$&$25=5^{2}$& $x^{2}+x+2$&$27=3^{3}$& $x^{3}+2x^{2}+x+1$ \\
$32=2^{5}$&$x^{5}+x^{3}+1$&$49=7^{2}$&$x^{2}+x+3$&$64=2^{6}$&$x^{6}+x^{4}+x^{3}+1$\\
$81=3^{4}$&$x^{4}+x+2$&$121=11^{2}$&$x^{2}+4x+2$&$125=5^{3}$&$x^{3}+3x+2$\\
$128=2^{7}$&$x^{7}+x+1$&$169=13^{2}$&$x^{2}+x+2$&$243=3^{5}$&$x^{5}+2x+1$\\
$256=2^{8}$&$x^{8}+x^{4}+x^{3}+$&$289=17^{2}$&$x^{2}+x+3$&$343=7^{3}$&$x^{3}+3x+2$\\
&$x^{2}+1$&&&&\\
$361=19^{2}$&$x^{2}+x+2$&$512=2^{9}$&$x^{9}+x^{4}+1$&$529=23^{2}$&$x^{2}+2x+5$\\
$625=5^{4}$&$x^{4}+x^{2}+2x+2$&$729=3^{6}$&$x^{6}+x+2$&$841=29^{2}$&$x^{2}+24x+2$\\
$961=31^{2}$&$x^{2}+29x+3$&$1024=2^{10}$&$x^{10}+x^{6}+x^{5}+$&$1331=11^{3}$&$x^{3}+2x+9$\\
&&&$x^{3}+x^{2}+x+1$&&\\
$1369=37^{2}$&$x^{2}+33x+2$&$1681=41^{2}$&$x^{2}+38x+6$ &$1849=43^{2}$&$x^{2}+x+3$\\
$2048=2^{11}$&$x^{11}+x^{2}+1$&$2187=3^{7}$&$x^{7}+x^{2}+2x+1$&$2197=13^{3}$&$x^{3}+x^{2}+7$\\
$2209=47^{2}$&$x^{2}+x+13$&$2401=7^{4}$&$x^{4}+5x^{2}+4x+3$&$2809=53^{2}$&$x^{2}+49x+2$\\
$3125=5^{5}$&$x^{5}+4x+2$&$3481=59^{2}$&$x^{2}+58x+2$&$3721=61^{2}$&$x^{2}+60x+2$\\
$4096=2^{12}$&$x^{12}+x^{8}+x^{2}+$&$4489=67^{2}$&$x^{2}+63x+2$&$4913=17^{3}$&$x^{3}+x+14$\\
&$x+1$&&&&\\
$5041=71^{2}$&$x^{2}+69x+7$&$5329=73^{2}$&$x^{2}+70x+5$&$6241=79^{2}$&$x^{2}+78x+3$\\
$6561=3^{8}$&$x^{8}+2x^{5}+x^{4}+$&$6859=19^{3}$&$x^{3}+4x+17$&$6889=83^{2}$&$x^{2}+82x+2$\\
&  $2x^{2}+2x+2$ &&&&\\
$7921=89^{2}$&$x^{2}+82x+3$&$8192=2^{13}$&$x^{13}+x^{4}+x^{3}+$&$9409=97^{2}$&$x^{2}+96x+5$\\
&&&$x+1$&\\
$10201=101^{2}$&$x^{2}+97x+2$&$10609=103^{2}$&$x^{2}+102x+5$&$11449=107^{2}$&$x^{2}+103x+2$\\
$11881=109^{2}$&$x^{2}+108x+6$&$12167=23^{3}$&$x^{3}+2x+18$&$12769=113^{2}$&$x^{2}+101x+3$\\
$14641=11^{4}$&$x^{4}+8x^{2}+10x+2$&$15625=5^{6}$&$x^{6}+x^{4}+4x^{3}+$&$16129=127^{2}$&$x^{2}+126x+3$\\
&&&$x^{2}+2$&\\
$16384=2^{14}$&$x^{14}+x^{7}+x^{5}+$&$16807=7^{5}$&$x^{5}+x+4$&$17161=131^{2}$&$x^{2}+83x+127$\\
&$x^{3}+1$&&&\\
$18769=137^{2}$&$x^{2}+95x+20$&$19321=139^{2}$&$x^{2}+72x+111$&$19683=3^{9}$&$x^{9}+2x^{3}+2x^{2}+$\\
&&&&&$x+1$\\
$22201=149^{2}$&$x^{2}+144x+34$&$22801=151^{2}$&$x^{2}+16x+51$&$24389=29^{3}$&$x^{3}+2x+27$\\
$24649=157^{2}$&$x^{2}+153x+82$&$26569=163^2$&$x^2+153x+82$&$27889=167^2$&$x^2+156x+90$\\
$28561=13^4$&$x^4+3x^2+12x+2$&$29791=31^3$&$x^3+x+28$&$29929=173^2$&$x^2+80x+104$\\
$32041=179^2$&$x^2+67x+154$&$32761=181^2$&$x^2+11x+78$&$32768=2^{15}$&$x^{15}+x+1$\\
 \hline
\end{tabular}
}
\end{table*}

\subsection{Starting points of lexiarcs in
$\mathrm{PG}(2,q)$,
 $q$ prime}\label{subsec_start_lexiarc}
\begin{proposition}\label{prop3_init_part_lexi}
Let $q$ be a prime. Then the $j$-th point of a lexiarc in
$\mathrm{PG}(2,q)$ is the same for all $q\ge q_{0}(j)$ where
$q_{0}(j)$ is large enough.
\end{proposition}
\begin{proof}
Suppose that, in \eqref{eq3_FOPalgorithm}, at a certain step
$j$ we have $K^{(j)}\setminus K^{(j-1)}=\{P\}$, with  $P=A_s$.
A point $Q=A_r \not \in K^{(j)}$ will be the next chosen point
in the extension process if and only if all the points $A_i$
with $i \in [s+1,r-1]$ are covered by $K^{(j)}$. That is, for
any $i \in [s+1,r-1]$ at least one of the determinants of the
coordinates of the points $P_1,P_2,A_i$, with $P_1,P_2 \in
K^{(j)}$, is equal to zero modulo $q$. This can happen only for
two reasons:
\begin{enumerate}
\item   $\det(P_1,P_2,A_i)= 0$: we  say that  $A_i$ is
    ``absolutely" covered by $K^{(j)}$;
\item  $\det(P_1,P_2,A_i)= m\neq0$, with $m \equiv 0 \bmod
    q$.
\end{enumerate}
It is clear that, for $q$ large enough, $q$ does not divide any
of the possible $m=\det(P_1,P_2,A_i)$ and then, at least for
$j$ ``small",  the points covered are just the absolutely
covered points. Therefore, when $q$ is large enough the
lexiarcs  share a certain number of points.
\end{proof}

The values of $q_{0}(j)$ can be found with the help of
calculations based on the proof of Proposition
\ref{prop3_init_part_lexi}. Also, we can directly consider
lexiarcs  constructed by the algorithm FOP for a convenient
region of $q$.

\begin{example}\label{ex3}
Values of $q_{0}(j)$, $j\le 24$, together with the homogenous
coordinates $(a_{0},a_{1},a_{2})$ of the common points, are
given in Table B. So, for all prime $q\ge251$ the first 24
points of a lexiarc are as in Table B. Since
$t^{L}_{2}(2,251)=63$ \cite{BDFMP-ArXivFOP}, we know
$24/63\thickapprox38\%$ of complete lexiarc points for $q=251$.
For growing $q$ this percentage decreases (relatively slowly).
For instance, for $q\thickapprox270000$ it is
$\thickapprox20\%$.

\begin{table*}[htbp]
\begin{center}
\textbf{Table B.} The first 24 points of complete lexiarcs in
$\mathrm{PG}(2,q)$, $q$ prime

\renewcommand{\arraystretch}{0.9}
\begin{tabular}{@{}r|rrr|r|r|rrr|r|r|rrr|r@{}}\hline
 $j$&$a_{0}$ & $a_{1}$ &$a_{2}$&$q_{0}(j)$&$j$&  $a_{0}$ & $a_{1}$ &$a_{2}$  &
 $q_{0}(j)$&$j$ &$a_{0}$ & $a_{1}$ &$a_{2}$  &$q_{0}(j)$   \\ \hline
 1& 0& 0 & 1 & 2 & 2 & 0 & 1 & 0 & 2 & 3 & 1 & 0 & 0 & 2 \\
 4& 1& 1 & 1 & 2 & 5 & 1 & 2 & 3 & 5 & 6 & 1 & 3 & 2 & 7 \\
 7& 1& 4 & 5 &11 & 8 & 1 & 5 & 4 &11 & 9 & 1 & 6 & 8 &17 \\
10& 1& 7 &11 &23 &11 & 1 & 8 & 6 &31 &12 & 1 & 9 &13 &37 \\
13& 1&10 &12 &37 &14 & 1 &11 & 7 &41 &15 & 1 &12 &22 &137 \\
16& 1&13 &16 &79 &17 & 1 &14 &17 &101&18 & 1 &15 &21 &71 \\
19& 1&16 & 9 &229&20 & 1 &17 &14 &151&21 & 1 &18 &10 &199 \\
22& 1&19 & 27&239&23 & 1 &20 &18 &197&24 & 1 &21 &15 &251\\ \hline
\end{tabular}
\end{center}
\end{table*}
\end{example}

\section{Upper bounds on $t_{2}(2,q)$, $q\in L$}\label{sec_bounds}

For bounds, we use data on sizes $t^{L}_{2}(2,q)$, $q\in L$, of complete
lexiarcs in $\mathrm{PG}(2,q)$ obtained and collected in
\cite{BDFKMP-ArXiv2015FOP,ComputBound-Svetlog2014,BDFMP-ArXivFOP,BDFMP-JG2015,BDMP-Bulg2012b,BFMPD-ENDM2013}
and in Tables 1 -- 6 of this work, see Appendix. Also, we use data on sizes $t_{2}^G(2,q)$, $q\in Q$, collected in \cite{BDFKMP-ArXiv2015Greedy}. Here $L$ and $Q$ are given by \eqref{eq1_L} 
and~\eqref{eq1_Q}, respectively. 

In Table 1, for $4\leq q\leq 316969$, $q$ non-prime, the sizes
 $t^{L}_{2}(2,q)$ (for short $t^{L}_{2}$)
 of complete lexiarcs in $\mathrm{PG}(2,q)$ are collected.

In Table 2,  for 15 sporadic $q$'s in the interval
$[323761\ldots430007]$, the sizes
 $t^{L}_{2}=t^{L}_{2}(2,q)$
 of complete lexiarcs in $\mathrm{PG}(2,q)$ are written.
 All $q\in L_{2}$, see~\eqref{eq1_L2}, are included in this table.

 In Tables 3 -- 6,  for \textbf{all prime powers} $q\le321007$ 
  the sizes $t^{L}_{2}=t^{L}_{2}(2,q)$
 of complete lexiarcs in $\mathrm{PG}(2,q)$ are collected.

We define the functions $h^{L}(q)$ and $h^G(q)$ as follows.
\begin{align}\label{eq3_m2L}
t_{2}^{L}(2,q)=h^{L}(q)\sqrt{3q\ln q};\\
t_{2}^G(2,q)=h^G(q)\sqrt{3q\ln q}.\notag
\end{align}

  Figure \ref{fig_1}
shows the following: the conjectural upper bound
\eqref{eq1_Prob bounds} of \cite{BDFKMP-PIT2014} divided by
$\sqrt{3q\ln q}$ (the top dashed red curve); values
$h^{L}(q)=t^{L}_{2}(2,q)/\sqrt{3q\ln q}$ for complete lexiarcs\smallskip 
~in $\mathrm{PG}(2,q)$, $q\in L$ of \eqref{eq1_L} (the 2-nd solid black curve);
values $h^G(q)=t^G_{2}(2,q)/\sqrt{3q\ln q}$
for complete arcs in $\mathrm{PG}(2,q)$ collected in \cite{BDFKMP-ArXiv2015Greedy},
$q\in Q$ of \eqref{eq1_Q} (the bottom solid blue curve); upper bounds  ``1.05" and ``0.998" (the horizontal dashed red
lines).  Vertical dashed-dotted magenta lines mark regions
$q\le160001$  and $q\le321007$ of the complete computer search for all  prime powers $q$.
\begin{figure}[htbp]
\includegraphics[width=\textwidth]{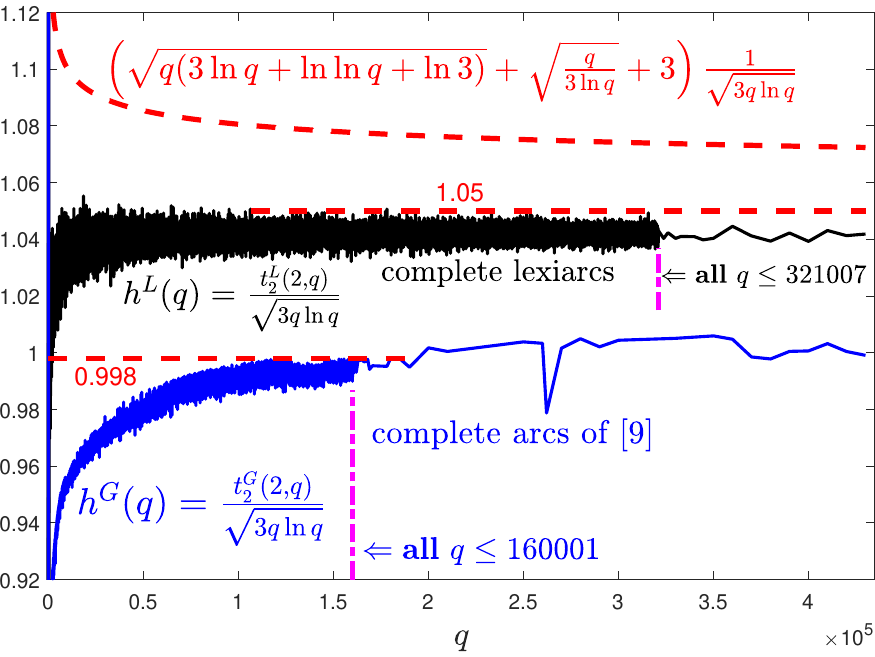}
\caption{\textbf{Conjectural upper bound \eqref{eq1_Prob bounds}
of \cite{BDFKMP-PIT2014}  vs sizes
$t^{L}_{2}(2,q)$ of complete lexiarcs and sizes $t^G_{2}(2,q)$ of  complete arcs 
collected in \cite{BDFKMP-ArXiv2015Greedy} (all values are divided by $\sqrt{3q\ln q}$):} conjectural upper bound \eqref{eq1_Prob bounds}
of \cite{BDFKMP-PIT2014} divided by $\sqrt{3q\ln q}$) (\emph{top dashed red
curve}); values $h^{L}(q)=t^{L}_{2}(2,q)/\sqrt{3q\ln q}$ for complete lexiarcs in\smallskip
~$\mathrm{PG}(2,q)$, $q\in L$ (\emph{the 2-nd solid black
curve}); values $h^G(q)=t^G_{2}(2,q)/\sqrt{3q\ln q}$
for complete arcs  collected in \cite{BDFKMP-ArXiv2015Greedy}, $q\in Q$ (\emph{bottom solid blue
curve}); upper bounds ``1.05", ``0.998" (\emph{horizontal dashed red lines}).  \emph{Vertical
dashed-dotted magenta lines} mark regions $q\le160001$ and $q\le321007$ of the complete computer search
for all prime powers $q$}
\label{fig_1}
\end{figure}

Figure \ref{fig_2} shows the upper bound $1.05\sqrt{3q\ln q}$
 (the top dashed-dotted red curve), see \eqref{eq1_Bounds1L},
 sizes $t^{L}_{2}(2,q)$ of complete lexiarcs in
$\mathrm{PG}(2,q)$, $q\in L$ of \eqref{eq1_L} (the 2-nd solid black curve), and sizes $t^G_{2}(2,q)$
for complete arcs collected in \cite{BDFKMP-ArXiv2015Greedy}, $q\in Q$  of \eqref{eq1_Q} (the bottom solid blue
curve).  Vertical
dashed-dotted magenta lines mark regions $q\le160001$ and $q\le321007$.

The curves of $1.05\sqrt{3q\ln q}$ and $t^{L}_{2}(2,q)$ almost
coalesce with each other in the scale of Figure \ref{fig_2}. It should be noted that $1.05\sqrt{3q\ln q}>t_{2}^{L}(2,q)$ if $q\ge106627$ and $q\ne178169$. In the other side, we have 
$1.05\sqrt{3q\ln q}>t_{2}^{G}(2,q)$ for all $q\in Q$, including $q=178169$, see \eqref{eq1_Q1},~\eqref{eq1_Q}. So, in fact, $1.05\sqrt{3q\ln q}$ is the upper bound on $t_2(2,q)$ for all $q\in L$.
\begin{figure}[htbp]
\includegraphics[width=\textwidth]{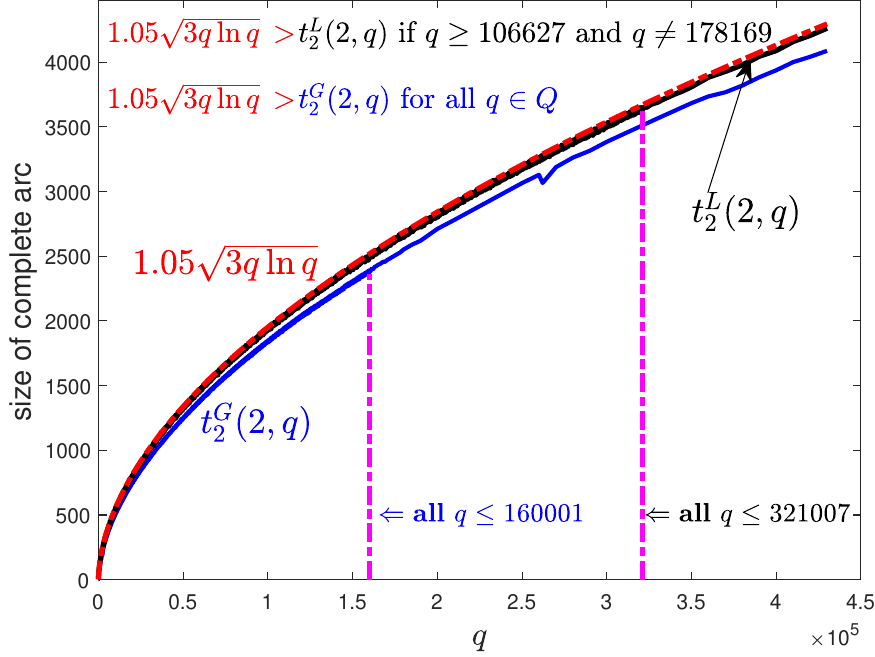}
\caption{\textbf{Upper bound $1.05\sqrt{3q\ln q}$
vs sizes
$t^{L}_{2}(2,q)$ of complete lexiarcs and sizes $t^G_{2}(2,q)$ of  complete arcs 
collected in \cite{BDFKMP-ArXiv2015Greedy}:}
upper bound  $1.05\sqrt{3q\ln q}$
(\emph{top dashed-dotted red curve}); sizes $t^{L}_{2}(2,q)$
of complete lexiarcs in $\mathrm{PG}(2,q)$, $q\in L$
(\emph{the 2-nd solid black curve}); sizes $t^G_{2}(2,q)$
for complete arcs collected in \cite{BDFKMP-ArXiv2015Greedy}, $q\in Q$ (\emph{bottom solid blue
curve}).  \emph{Vertical
dashed-dotted magenta lines} mark regions $q\le160001$ and $q\le321007$}
\label{fig_2}
\end{figure}

Figure \ref{fig_3} presents the difference $1.05\sqrt{3q\ln
q}-t^{G}_{2}(2,q)$ between the upper bound\linebreak
$1.05\sqrt{3q\ln q}$ and the sizes $t^{G}_{2}(2,q)$ of complete
arcs in $\mathrm{PG}(2,q)$ collected in \cite{BDFKMP-ArXiv2015Greedy}, $q\in Q$  of~\eqref{eq1_Q}. The  vertical
dashed-dotted magenta line marks the region $q\le160001$ of the complete computer search. The ``tooth" on the graphics corresponds to $\overline{t}_{2}(2,2^{18})=3066$  \cite{DGMP-Innov}.
\begin{figure}[htbp]
\includegraphics[width=\textwidth]{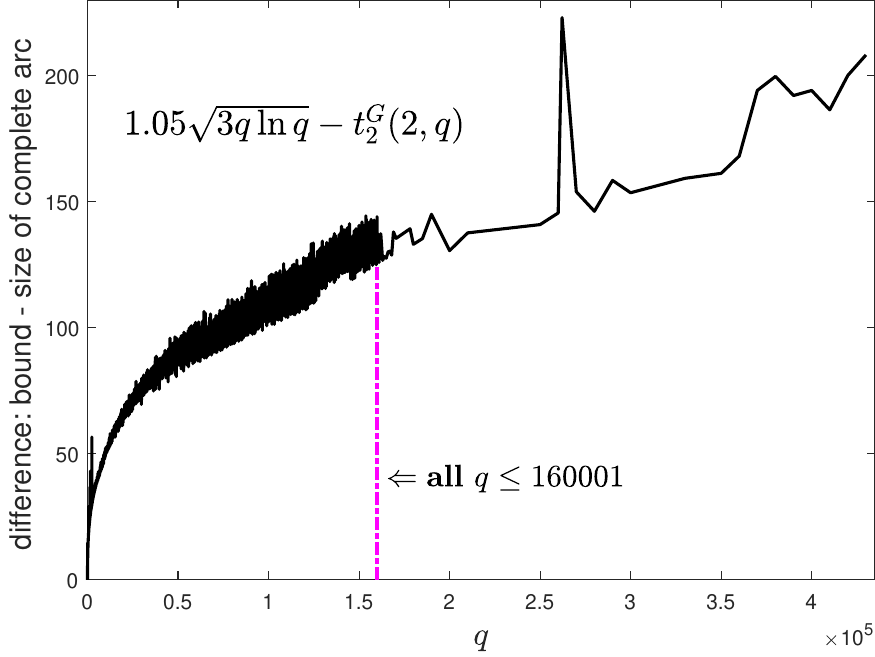}
\caption{\textbf{Difference $1.05\sqrt{3q\ln
q}-t^{G}_{2}(2,q)$ between upper bound
$1.05\sqrt{3q\ln q}$ and sizes $t^{G}_{2}(2,q)$ of complete
arcs in $\mathrm{PG}(2,q)$, collected in \cite{BDFKMP-ArXiv2015Greedy}, $q\in Q$}.  \emph{Vertical
dashed-dotted magenta line} marks region $q\le160001$ of the complete computer search}
\label{fig_3}
\end{figure}
\begin{observation}\label{obs2_1.05}
The difference $1.05\sqrt{3q\ln
q}-t^{G}_{2}(2,q)$ between the upper bound\linebreak
$1.05\sqrt{3q\ln q}$ and sizes $t^{G}_{2}(2,q)$ of complete
arcs in $\mathrm{PG}(2,q)$, collected in \emph{\cite{BDFKMP-ArXiv2015Greedy}}, tends to increase when $q$ grows.
\end{observation}

In Figure \ref{fig_4}, the conjectural upper bound from
\cite{BDFKMP-PIT2014} $\sqrt{q(3\ln q+\ln\ln
q+\ln 3)}+\sqrt{\frac{q}{3\ln q}}+3$, see~\eqref{eq1_Prob bounds}
(the top dashed-dotted red curve) and the sizes
$t^{L}_{2}(2,q)$ of complete lexiarcs in $\mathrm{PG}(2,q)$,
$q\in L$ of \eqref{eq1_L} (the bottom solid black curve) are shown. The vertical
dashed-dotted magenta line marks region $q\le321007$.
\begin{figure}[htbp]
\includegraphics[width=\textwidth]{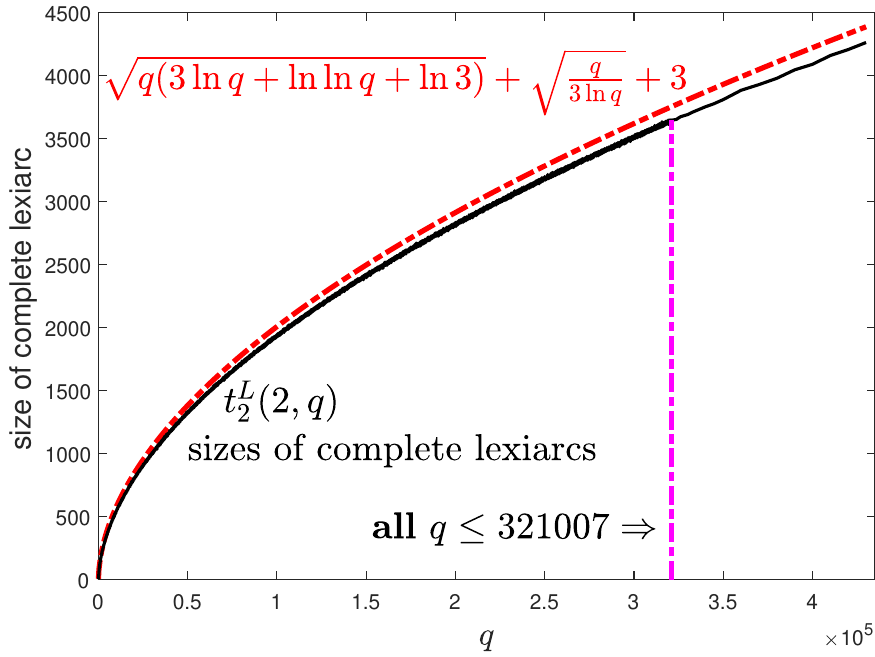}
\caption{\textbf{Conjectural upper bound \eqref{eq1_Prob bounds}
of \cite{BDFKMP-PIT2014} vs sizes $t^{L}_{2}(2,q)$ of complete lexiarcs in $\mathrm{PG}(2,q)$:}
conjectural upper bound  of \cite{BDFKMP-PIT2014}
(\emph{top dashed-dotted red curve}); sizes $t^{L}_{2}(2,q)$ of complete lexiarcs in $\mathrm{PG}(2,q)$, $q\in L$
(\emph{bottom solid black curve}).  \emph{Vertical
dashed-dotted magenta line} marks region $q\le321007$ of the complete computer search}
\label{fig_4}
\end{figure}

Figure \ref{fig_5} presents the difference
$$
\left(\sqrt{q(3\ln q+\ln\ln q+\ln3)}+\sqrt{\frac{q}{3\ln q}}+3\right)-t^{L}_{2}(2,q)
$$
 between the conjectural upper bound
\eqref{eq1_Prob bounds} from \cite{BDFKMP-PIT2014}
 and the sizes
$t^{L}_{2}(2,q)$ of complete lexiarcs in $\mathrm{PG}(2,q)$.
\begin{figure}[htbp]
\includegraphics[width=\textwidth]{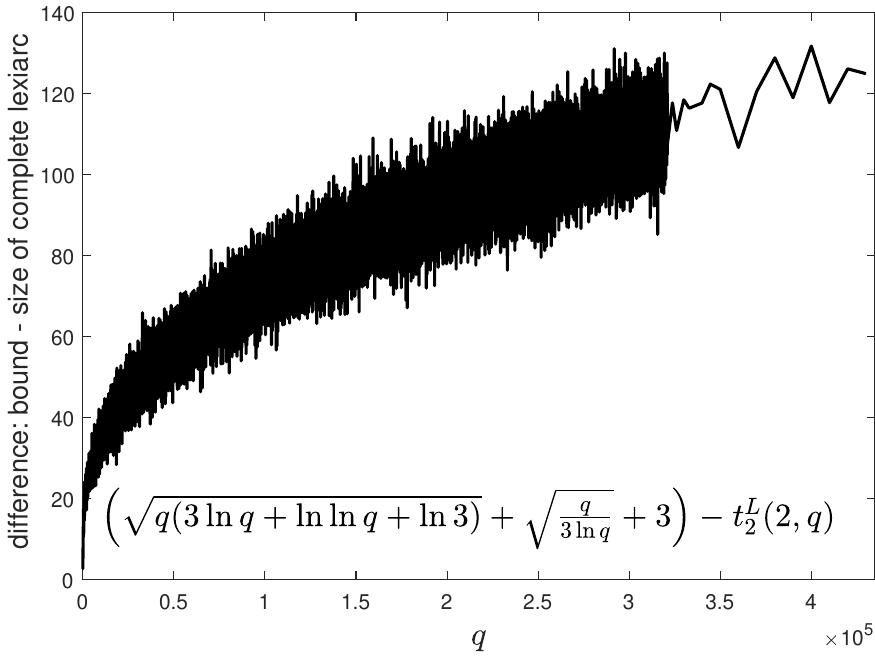}
\caption{\textbf{Difference $\left(\sqrt{q(3\ln q+\ln3\ln q+\ln3)}+\sqrt{\frac{q}{3\ln q}}+3\right)-t^{L}_{2}(2,q)
$ between the conjectural upper bound
\eqref{eq1_Prob bounds} from \cite{BDFKMP-PIT2014}
 and the sizes
$t^{L}_{2}(2,q)$ of complete lexiarcs in $\mathrm{PG}(2,q)$, $q\in L$}}
\label{fig_5}
\end{figure}

\begin{observation}\label{obs2_diff conj}
The difference  $\left(\sqrt{q(3\ln q+\ln\ln q+\ln3)}+\sqrt{\frac{q}{3\ln q}}+3\right)-t^{L}_{2}(2,q)
$ between the conjectural upper bound
\eqref{eq1_Prob bounds} from \emph{\cite{BDFKMP-PIT2014}}
 and the sizes
$t^{L}_{2}(2,q)$ of complete lexiarcs in $\mathrm{PG}(2,q)$ tends to increase when $q$ grows.
\end{observation}

Data for $q\in L$ of \eqref{eq1_L}, collected in Tables 1 -- 6, give rise to
Theorems \ref{th3_m_lexi_upper} and \ref{th3_m_lexi_lower} on upper and lower bounds on $h^{L}(q)$.

\begin{theorem}\label{th3_m_lexi_upper}
 Let the functions\/ $h(q)$ and $h^{L}(q)$ be as
in \eqref{eq1_h(q)} and \eqref{eq3_m2L}. Let $L_{2}$ be the set
of values of $q$ given by \eqref{eq1_L2} and Table~\emph{2}. In
$\mathrm{PG}(2,q)$ the following upper bounds are provided by
complete lexiarcs:
\begin{align*}
h(q)\le h^{L}(q)<
\left\{\begin{array}{lcl}
1.056&\text{for}&\phantom{1781}16\le q\le 18443\\
1.053&\text{for}&\phantom{1}18443< q\le 80407\\
1.051&\text{for}&\phantom{1}80407< q\le 178169\\
1.050&\text{for}&178169< q\le 315701\\
1.048&\text{for}&315701< q\le 321007\mbox{ and }q\in L_{2}
\end{array}
\right.;
\end{align*}
\begin{align}\label{eq3_1.05lex}
h(q)\le h^{L}(q)<1.05\text{ for }106627\le q\le321007\mbox{ and }q\in L_{2},~~q\neq178169.
\end{align}
\end{theorem}

\begin{theorem}\label{th3_m_lexi_lower}
 Let
the function $h^{L}(q)$ be as in \eqref{eq3_m2L}. Let $L_{2}$
be the set of values of $q$ given by \eqref{eq1_L2} and
Table~\emph{2}. In $\mathrm{PG}(2,q)$ the following lower
bounds on $h^{L}(q)$ hold.
\begin{align*}
 h^{L}(q)>
\left\{\begin{array}{lcl}
1.023&\text{for}&\phantom{1}11971\le q\le 34583\\
1.028&\text{for}&\phantom{1} 34583< q\le 70451\\
1.032&\text{for}&\phantom{1}70451< q\le 159349\\
1.033&\text{for}&159349<q\le 192133\\
1.034&\text{for}&192133< q\le 297967\\
1.036&\text{for}&297967< q\le 321007\mbox{ and }q\in L_{2}\
\end{array}
\right..
\end{align*}
\end{theorem}

Figure \ref{fig_6} illustrates Theorems
\ref{th3_m_lexi_upper}, \ref{th3_m_lexi_lower} and presents the
value $h^{L}(q)=t_{2}^{L}(2,q)/\sqrt{3q\ln q}$, $q\in L$ (the
solid black curve), and its upper and lower bounds (the
dashed red lines) in more detail than Figure
\ref{fig_1}. Also, the middle line
$y=h^{L}_{\mbox{mid}}=1.042$ is shown. The percentage for a
bound $B$ is calculated as
\begin{align}\label{eq3_percentage}
\frac{B-h^{L}_{\mbox{mid}}}{h^{L}_{\mbox{mid}}}100\%,\quad h^{L}_{\mbox{mid}}=1.042,
\end{align}
where
$B\in\{1.056,1.053,1.051,1.050,1.048,1.023,1.028,1.032,1.033,1.034,1.036\}$.
\begin{figure}[htbp]
\includegraphics[width=\textwidth]{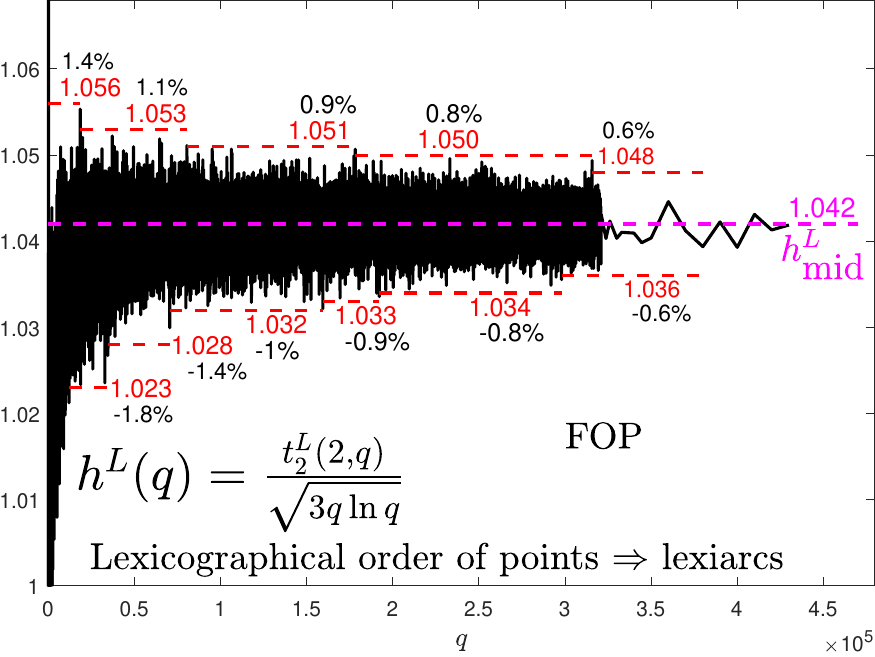}
\caption{\textbf{Values of
$h^{L}(q)=t^{L}_{2}(2,q)/\sqrt{3q\ln q}$ for complete lexiarcs in
$\mathrm{PG}(2,q)$, $q\in L$ (\emph{solid black curve}) and
the corresponding upper and lower bounds (\emph{dashed red lines})}.
\emph{Magenta line} $y=1.042$ is a ``middle'' line for $h^{L}(q)$}
\label{fig_6}
\end{figure}

By Theorems \ref{th3_m_lexi_upper}, \ref{th3_m_lexi_lower} and
Figure \ref{fig_6}, we have the following observation.

\begin{observation}
\label{obs2_m_lexi} For $q\in L$, $q\ge11971$,  the values of
$h^{L}(q)=t_{2}^{L}(2,q)/\sqrt{3q\ln q}$ oscillate around the
horizontal line $y=1.042$ with a small amplitude. For growing
$q$, the oscillation amplitude decreases. Upper bounds on the
amplitude decrease from $1.4\%$ to $0.6\%$ while lower bounds
change from $-1.8\%$ to $-0.6\%$, where the percentage
corresponds to \eqref{eq3_percentage}. See also  Remark
\emph{\ref{rem4_enigma}}.
\end{observation}

\begin{remark}\label{rem4_enigma}\begin{description}
                                   \item[(i)] The oscillation (with decreasing amplitude) of
 $h^{L}(q)$
around the  horizontal line  (see Figure
\ref{fig_6} and Observation
\ref{obs2_m_lexi}) is an
interesting \textbf{enigma} that should be investigated and
explained.
\item[(ii)] It would be  interesting to understand
       the working mechanism and to do quantitative estimates
       for
the  step-by-step algorithm FOP, see \eqref{eq3_FOPalgorithm},
similarly to the work \cite{BDFKMP-PIT2014} where the working
mechanism of a greedy algorithm is treated.
                                 \end{description}
\end{remark}

\begin{proposition}\label{cor3} 
Let $t_{2}(2,q)$ be the smallest size of a complete arc in the
projective plane $\mathrm{PG}(2,q)$. 
The following upper
bound holds:
\begin{align*}t_{2}(2,q)<1.05 \sqrt{3q\ln q}~~\mbox{ for }~~7\le q\le321007,~q\in L_{2}.
\end{align*}
\end{proposition}
\begin{proof}
By \cite[Tab.\,3]{BDFKMP-ArXiv2015Greedy}, we have
$t_{2}^G(2,178169)=2530$, whence
$h(178169)<h^G(178169)<0.996$. Now the assertion
follows from \eqref{eq3_1.05lex} and \eqref{eq1_Bounds1G}.
\end{proof}

\section{On the common nature of lexiarcs and random arcs}
\label{sec-random} In this section we compare complete random
arcs in $\mathrm{PG}(2,q)$ and complete lexiarcs. The random
arcs are constructed iteratively. The next point of an
incomplete current arc is taken randomly among points that are
not covered by the arc bisecants. Let $t_{2}^{R}(2,q)$ be the
size of a complete random arc in $\mathrm{PG}(2,q)$. The values
of $t_{2}^{R}(2,q)$ for $q\le46337$, $q$ prime, are collected
in the work \cite{BDFMP-ArXivRandom}. We define a function
$h^{R}(q)$ by
\begin{align*}
t_{2}^{R}(2,q)=h^{R}(q)\sqrt{3q\ln q}.
\end{align*}

Values of $h^{R}(q)=t^{R}_{2}(2,q)/\sqrt{3q\ln q}$ (the solid
green curve) and upper bound $y=1.054$ (the dashed-dotted red
line) for $q\le46337$, $q$ prime, are shown in Figure
\ref{fig_7}.
\begin{figure}[htbp]
\includegraphics[width=\textwidth]{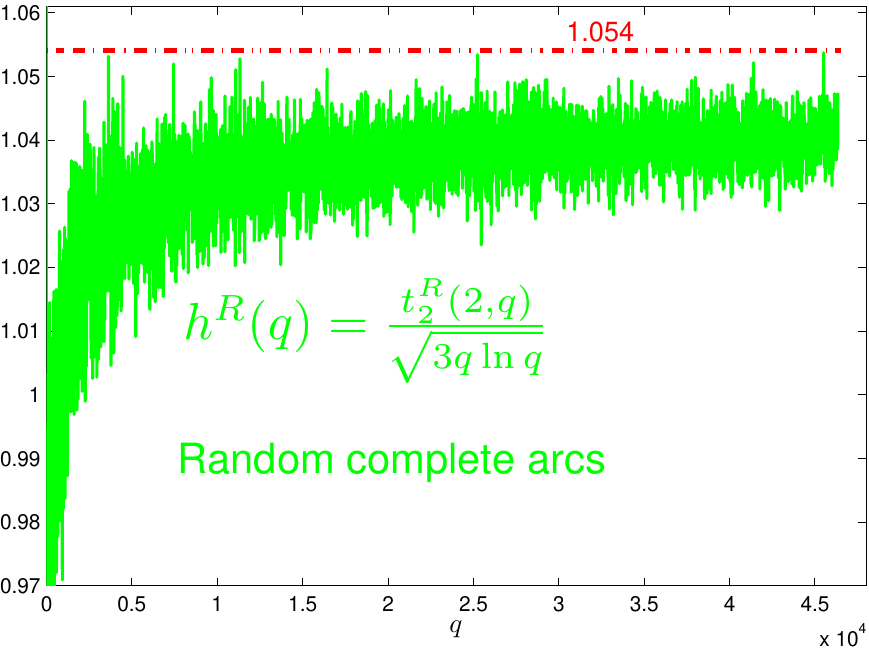}
\caption{\textbf{Values of
$h^{R}(q)=t^{R}_{2}(2,q)/\sqrt{3q\ln q}$ for complete random arcs in
$\mathrm{PG}(2,q)$, $q\le46337$, $q$ prime} (\emph{solid green curve}) and
the upper bound $y=1.054$ (\emph{dashed-dotted red line})}
\label{fig_7}
\end{figure}

It is useful to compare sizes of complete lexiarc and
 complete random arcs. The percentage difference
 $\Delta^{LR}(q)$
between sizes of complete lexiarcs and complete random arcs in
$\mathrm{PG}(2,q)$ for $q\le46337$, $q$ prime, is shown in
Figure \ref{fig_8} in the form
$$
\Delta^{LR}(q)=\frac{t_{2}^{L}(2,q)-t_{2}^{R}(2,q)}{t_{2}^{L}(2,q)\vphantom{H^{H{^H}}}}100\%
.$$
\begin{figure}[htbp]
\includegraphics[width=\textwidth]{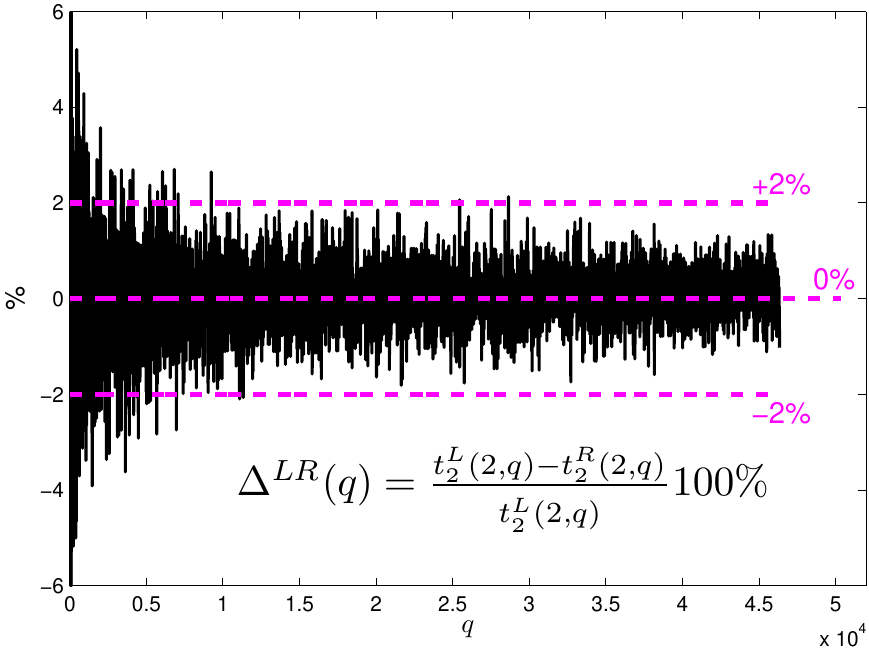}
\caption{\textbf{Percentage difference $\Delta^{LR}(q)$ between sizes
$t^{L}_{2}(2,q)$ of complete lexiarcs and sizes $t_{2}^{R}(2,q)$
of complete random arcs in $\mathrm{PG}(2,q)$}}
\label{fig_8}
\end{figure}

 One can see that the difference $\Delta^{LR}(q)$ is relatively small; it is in
the region $\thickapprox\pm2\%$. Moreover, the difference
$\Delta^{LR}(q)$ oscillates around the horizontal line $y=0$.
It means that, perhaps, \textbf{lexiarcs and random arcs have
the same nature.} This is expected, as the lexicographical
order of points is a random order in the geometrical sense.

\section{List of tables with sizes of complete lexiarcs in the projective plane
 $\mathrm{PG}(2,q)$}\label{sec_list tables}

\indent\textbf{Table 1.} The  sizes
 $t^{L}_{2}=t^{L}_{2}(2,q)$
 of complete lexiarcs in planes
 $\mathrm{PG}(2,q),$\\ $4\leq q\leq 316969$, $q$ non-prime. \textbf{p. 30}

 \textbf{Table 2.} The  sizes
 $t^{L}_{2}=t^{L}_{2}(2,q)$
 of complete lexiarcs in planes
 $\mathrm{PG}(2,q),$\\ $323761\leq q\leq 430007$, sporadic
 $q$. \textbf{p. 30}

\textbf{Table 3.} The sizes
 $t_{2}^{L}=t_{2}^{L}(2,q)$
 of complete lexiarcs in planes
 $\mathrm{PG}(2,q),$\\ $3\leq q< 10000$, $q$ prime power. \textbf{pp. 31--34}

 \textbf{Table 4.} The sizes
 $t_{2}^{L}=t_{2}^{L}(2,q)$
 of complete lexiarcs in planes
 $\mathrm{PG}(2,q),$\\ $10000< q< 100000$, $q$ 
 prime power. \textbf{pp. 35--56}

 \textbf{Table 5.} The sizes
 $t_{2}^{L}=t_{2}^{L}(2,q)$
 of complete lexiarcs in planes
 $\mathrm{PG}(2,q),$\\ $100000< q\leq 301813$, $q$ 
 prime power. \textbf{pp. 57--106}

 \textbf{Table 6.} The sizes
 $t_{2}^L=t_{2}^L(2,q)$
 of complete lexiarcs in planes
 $\mathrm{PG}(2,q),$\\ $301813< q\leq 321007$, $q$ prime power. \textbf{pp. 107--111}

\section{Conclusion} \label{sec_Concl}

This work contains tables of sizes of \textbf{small
complete arcs} in the projective plane
 $\mathrm{PG}(2,q)$ for \textbf{all} $q\le321007$, $q$ prime power, and  15 sporadic $q$'s in the interval $[323761\ldots$ $430007]$. These arcs are obtained with the
 help of a \textbf{step-by-step algorithm with fixed order of points
(FOP)}, see
\cite{BDFKMP-ArXiv2015FOP,BDFKMP-JGtoappear,ComputBound-Svetlog2014,BDFMP-JG2015,BDMP-Bulg2012b,BFMPD-ENDM2013,BDFMP-ArXivFOP}.
The algorithm FOP fixes a particular order on points of the
projective plane $\mathrm{PG}(2,q)$. In each step, the
algorithm FOP adds to an incomplete current arc the next point
in this order  not lying on bisecants of this arc. For arcs,
sizes of which are collected in this work, a
\textbf{lexicographical order of points} in the algorithm FOP
was used. Therefore these arcs are called
\textbf{\emph{lexiarcs}}.

In this work, \textbf{upper bounds on the smallest size
$t_{2}(2,q)$ of a complete arc in the projective plane
$\mathrm{PG}(2,q)$} are considered on the base of the sizes of
complete lexiarcs, collected in this work, and of the smallest
known (up to June 2015) sizes of complete arcs in $\mathrm{PG}(2,q)$ for
\textbf{all} $q\le160001$, $q$ prime power, collected in
\cite{BDFKMP-ArXiv2015Greedy}.

For $q\le321007$, the \emph{computer search}, the results of
which are used in this work, is \emph{complete}, i.e.\ it has
been performed for \emph{all} $q$ prime powers. This
\emph{proves} that the described upper bound  $t_{2}(2,q)<1.05\sqrt{3q\ln q}$
is valid, at least, in this region, see \eqref{eq1_Bounds1L} of Theorem~\ref{th1_main} and Figure \ref{fig_1}. Calculations
executed for sporadic $q\le430007$ strengthen  the confidence
in the validity of these bounds for large values of $q$, see also Figures
\ref{fig_1}, \ref{fig_2}, \ref{fig_3}, \ref{fig_6}. Moreover, the bounds of
Theorem~\ref{th1_main}
 are close to the
conjectural bounds of \cite{BDFKMP-PIT2014} cited in
 Conjecture~\ref{conj1}, see Figures~\ref{fig_4},~\ref{fig_5}. By all the arguments, we conjecture that
 the
bound $t_{2}(2,q)<1.05\sqrt{3q\ln q}$ holds for all
$q$, see Conjecture \ref{conj1_for all q}.

The most of the smallest known complete arcs, the sizes of which
are collected in \cite{BDFKMP-ArXiv2015Greedy}, are obtained by
computer search using randomized greedy algorithms \cite
{BDFKMP-JGtoappear,BDFMP-DM,BDFMP-JG2013,BDFMP-JG2015,DFMP-JG2005,DFMP-JG2009,DMP-JG2004}.
In each step, a step-by-step greedy algorithm adds to
    an incomplete current arc  a point providing the maximum
    possible (for the given step) number of new covered
    points.

Complete arcs obtained by greedy algorithms have smaller sizes
than complete lexiarcs, however the greedy algorithms take
essentially greater computer time than the algorithm FOP. This
is why the complete computer search for all  prime powers $q$
with the help of greedy algorithms is done for $q\le160001$
\cite{BDFKMP-ArXiv2015Greedy} whereas the complete search by
algorithm FOP is executed for $q\le321007$.

On the other hand, the difference between sizes of complete
lexiarcs and the smallest known sizes of complete arcs in
$\mathrm{PG}(2,q)$ is relatively small; it is $\lesssim6\%$ for
 $q\gtrapprox90000$, see \cite{BDFKMP-ArXiv2015FOP}. Therefore, \emph{for the computer
search with large $q$} the \emph{algorithm FOP} using the
lexicographical order of the points seems to be \emph{better
than greedy algorithms}.

Moreover, investigations of complete lexiarcs for large $q$
could help to understand the \textbf{enigma} connected with
oscillation (with decreasing amplitude) of the values $h^{L}(q)$
around a horizontal line (see Figure
\ref{fig_6}, Observation
\ref{obs2_m_lexi}, and Remark
\ref{rem4_enigma}).

It would be useful also to understand the structure of
lexiarcs, in particular, the initial part of a lexiarc that is
the same for all lexiarcs with greater $q$, see Subsection
\ref{subsec_start_lexiarc}.

Also, it would be  interesting to understand
       the working mechanism and to do quantitative estimates
       for the  step-by-step algorithm FOP, see \eqref{eq3_FOPalgorithm},
similarly to the work \cite{BDFKMP-PIT2014} where the working
mechanism of a greedy algorithm is treated.

Finally, further investigations of random arcs and their
``similarity'' to lexiarcs, see Figures~\ref{fig_7} and
\ref{fig_8}, would be useful.

This work can be considered as a continuation and development
of the paper \cite{BDFKMP-JGtoappear}.

\section{Appendix. Tables sizes of complete lexiarcs in the projective
plane $\mathrm{PG}(2,q)$}

\normalsize{\textbf{Table 1.} The  sizes
 $t^{L}_{2}=t^{L}_{2}(2,q)$
 of complete lexiarcs in planes
 $\mathrm{PG}(2,q),$ $4\leq q\leq 316969$,\\ $q$ non-prime\medskip}\\
\small{
\renewcommand{\arraystretch}{1.00}
\begin{tabular}
 {@{}rr|rr|rr|rr|rr|rr@{}}
  \hline
$q$ & $t^{L}_{2}$ &
$q$ & $t^{L}_{2}$ &
$q$ & $t^{L}_{2}$ &
$q$ & $t^{L}_{2}$ &
$q$ & $t^{L}_{2}$ &
$q$ & $t^{L^{\vphantom{H^{L}}}}_{2}$ \\
 \hline
 4 & 6  &  8 & 6  &  9 & 8  &  16 & 12  &  25 & 14  &  27 & 14  \\
 32 & 17  &  49 & 24  &  64 & 26  &  81 & 31  &  121 & 40  &  125 & 40  \\
 128 & 43  &  169 & 49  &  243 & 59  &  256 & 64  &  289 & 68  &  343 & 77  \\
 361 & 80  &  512 & 99  &  529 & 97  &  625 & 108  &  729 & 120  &  841 & 129  \\
 961 & 140  &  1024 & 142  &  1331 & 173  &  1369 & 172  &  1681 & 195  &  1849 & 207  \\
 2048 & 219  &  2187 & 227  &  2197 & 228  &  2209 & 228  &  2401 & 242  &  2809 & 261  \\
 3125 & 283  &  3481 & 296  &  3721 & 310  &  4096 & 328  &  4489 & 344  &  4913 & 362  \\
\hline
 \end{tabular}
}
\begin{center}
\normalsize{\textbf{Table 2.} The  sizes
 $t^{L}_{2}=t^{L}_{2}(2,q)$
 of complete lexiarcs in planes
 $\mathrm{PG}(2,q),$ $323761\leq q\leq 430007$, sporadic
 $q$\smallskip}

\small{
 \renewcommand{\arraystretch}{1.00}
 \begin{tabular}
 {cc|cc|cc|cc}
  \hline
$q$ & $t^{L}_{2}$ &
$q$ & $t^{L}_{2}$ &
$q$ & $t^{L}_{2}$ &
$q$ & $t_{2}^{L^{\vphantom{H^{L}}}}$ \\
 \hline
 323761 & 3652  &  326041 & 3673  &  330017 & 3690  &  332929 & 3710  \\
 340007 & 3752  &  344569 & 3775  &  350003 & 3809  &  360007 & 3883  \\
 370003 & 3928  &  380041 & 3978  &  390001 & 4045  &  400009 & 4089  \\
 410009 & 4159  &  420001 & 4206  &  430007 & 4262  &  &  \\
 \hline
 \end{tabular}
 }
\end{center}

\include{Table3_q_3_10000}
\include{Table4_q_10001_100000}
\include{Table5_q_100001_301813}
\include{Table6_q_301831_321007}

\end{document}

%% file: Table3_q_3_10000.tex
\normalsize{\textbf{Table 3.} The sizes
 $t_{2}^{L}=t_{2}^{L}(2,q)$
 of complete lexiarcs in planes
 $\mathrm{PG}(2,q),$ $3\leq q<10000$,\\ $q$ prime power
 \medskip}

\small{
 \renewcommand{\arraystretch}{0.97}
 \begin{tabular}
 {@{}c@{\,\,\,\,\,}c@{\,\,}|@{\,\,}c@{\,\,\,\,\,}c@{\,\,}|@{\,\,}c@{\,\,\,\,\,}c@{\,\,}|@{\,\,}c@{\,\,\,\,\,}c@{\,\,}|@{\,\,}c@{\,\,\,\,\,}c@{\,\,}|@{\,\,}c@{\,\,\,\,\,}c@{\,\,}|@{\,\,}c@{\,\,\,\,\,}c@{\,\,}|@{\,\,}c@{\,\,\,\,\,}c@{}}
  \hline
$q$ & $t_{2}^{L}$ &
$q$ & $t_2^L$ &
$q$ & $t_2^L$ &
$q$ & $t_2^L$ &
$q$ & $t_2^L$ &
$q$ & $t_2^L$ &
$q$ & $t_2^L$ &
$q$ & $t_{2}^{L^{\vphantom{H^{L}}}}$ \\
 \hline
 3 & 4  &  4 & 6  &  5 & 6  &  7 & 8  &  8 & 6  &  9 & 8  &  11 & 8  &  13 & 12  \\
 16 & 12  &  17 & 11  &  19 & 12  &  23 & 14  &  25 & 14  &  27 & 14  &  29 & 16  &  31 & 16  \\
 32 & 17  &  37 & 19  &  41 & 19  &  43 & 20  &  47 & 22  &  49 & 24  &  53 & 23  &  59 & 25  \\
 61 & 25  &  64 & 26  &  67 & 27  &  71 & 28  &  73 & 27  &  79 & 30  &  81 & 31  &  83 & 32  \\
 89 & 32  &  97 & 35  &  101 & 35  &  103 & 37  &  107 & 39  &  109 & 38  &  113 & 40  &  121 & 40  \\
 125 & 40  &  127 & 41  &  128 & 43  &  131 & 41  &  137 & 44  &  139 & 42  &  149 & 46  &  151 & 46  \\
 157 & 46  &  163 & 47  &  167 & 48  &  169 & 49  &  173 & 50  &  179 & 53  &  181 & 51  &  191 & 52  \\
\hline
 \end{tabular}
} \newpage

\normalsize{\textbf{Table 3.} Continue 1 }\medskip

\small{
 \renewcommand{\arraystretch}{1.00}
 \begin{tabular}
 {@{}c@{\,\,\,\,\,}c@{\,\,}|@{\,\,}c@{\,\,\,\,\,}c@{\,\,}|@{\,\,}c@{\,\,\,\,\,}c@{\,\,}|@{\,\,}c@{\,\,\,\,\,}c@{\,\,}|@{\,\,}c@{\,\,\,\,\,}c@{\,\,}|@{\,\,}c@{\,\,\,\,\,}c@{\,\,}|@{\,\,}c@{\,\,\,\,\,}c@{\,\,}|@{\,\,}c@{\,\,\,\,\,}c@{}}
  \hline
$q$ & $t_{2}^{L}$ &
$q$ & $t_{2}^{L}$ &
$q$ & $t_{2}^{L}$ &
$q$ & $t_{2}^{L}$ &
$q$ & $t_{2}^{L}$ &
$q$ & $t_{2}^{L}$ &
$q$ & $t_{2}^{L}$ &
$q$ & $t_{2}^{L^{\vphantom{H^{L}}}}$ \\
 \hline
 1907 & 209  &  1913 & 212  &  1931 & 215  &  1933 & 214  &  1949 & 213  &  1951 & 210  &  1973 & 217  &  1979 & 218  \\
 1987 & 220  &  1993 & 220  &  1997 & 216  &  1999 & 218  &  2003 & 218  &  2011 & 215  &  2017 & 224  &  2027 & 219  \\
 2029 & 217  &  2039 & 219  &  2048 & 219  &  2053 & 216  &  2063 & 223  &  2069 & 220  &  2081 & 220  &  2083 & 221  \\
 2087 & 223  &  2089 & 220  &  2099 & 226  &  2111 & 222  &  2113 & 229  &  2129 & 229  &  2131 & 228  &  2137 & 228  \\
 2141 & 229  &  2143 & 227  &  2153 & 226  &  2161 & 227  &  2179 & 226  &  2187 & 227  &  2197 & 228  &  2203 & 233  \\
 2207 & 230  &  2209 & 228  &  2213 & 231  &  2221 & 233  &  2237 & 234  &  2239 & 233  &  2243 & 234  &  2251 & 231  \\
 2267 & 233  &  2269 & 236  &  2273 & 238  &  2281 & 237  &  2287 & 235  &  2293 & 232  &  2297 & 234  &  2309 & 233  \\
\hline
 \end{tabular}

} \newpage

\normalsize{\textbf{Table 3.} Continue 2 }\medskip

\small{
 \renewcommand{\arraystretch}{1.00}
 \begin{tabular}
 {@{}c@{\,\,\,\,\,}c@{\,\,}|@{\,\,}c@{\,\,\,\,\,}c@{\,\,}|@{\,\,}c@{\,\,\,\,\,}c@{\,\,}|@{\,\,}c@{\,\,\,\,\,}c@{\,\,}|@{\,\,}c@{\,\,\,\,\,}c@{\,\,}|@{\,\,}c@{\,\,\,\,\,}c@{\,\,}|@{\,\,}c@{\,\,\,\,\,}c@{\,\,}|@{\,\,}c@{\,\,\,\,\,}c@{}}
  \hline
$q$ & $t_{2}^{L}$ &
$q$ & $t_{2}^{L}$ &
$q$ & $t_{2}^{L}$ &
$q$ & $t_{2}^{L}$ &
$q$ & $t_{2}^{L}$ &
$q$ & $t_{2}^{L}$ &
$q$ & $t_{2}^{L}$ &
$q$ & $t_{2}^{L^{\vphantom{H^{L}}}}$ \\
 \hline
 4423 & 339  &  4441 & 342  &  4447 & 344  &  4451 & 342  &  4457 & 346  &  4463 & 347  &  4481 & 343  &  4483 & 348  \\
 4489 & 344  &  4493 & 343  &  4507 & 346  &  4513 & 343  &  4517 & 345  &  4519 & 351  &  4523 & 349  &  4547 & 345  \\
 4549 & 349  &  4561 & 349  &  4567 & 348  &  4583 & 347  &  4591 & 349  &  4597 & 350  &  4603 & 352  &  4621 & 350  \\
 4637 & 352  &  4639 & 351  &  4643 & 349  &  4649 & 353  &  4651 & 351  &  4657 & 352  &  4663 & 351  &  4673 & 352  \\
 4679 & 359  &  4691 & 353  &  4703 & 361  &  4721 & 360  &  4723 & 354  &  4729 & 356  &  4733 & 356  &  4751 & 356  \\
 4759 & 355  &  4783 & 361  &  4787 & 359  &  4789 & 359  &  4793 & 360  &  4799 & 361  &  4801 & 362  &  4813 & 361  \\
 4817 & 357  &  4831 & 362  &  4861 & 363  &  4871 & 362  &  4877 & 356  &  4889 & 369  &  4903 & 361  &  4909 & 362  \\
\hline
 \end{tabular}

} \newpage

\normalsize{\textbf{Table 3.} Continue 3 }\medskip

\small{
 \renewcommand{\arraystretch}{1.00}
 \begin{tabular}
 {@{}c@{\,\,\,\,\,}c@{\,\,}|@{\,\,}c@{\,\,\,\,\,}c@{\,\,}|@{\,\,}c@{\,\,\,\,\,}c@{\,\,}|@{\,\,}c@{\,\,\,\,\,}c@{\,\,}|@{\,\,}c@{\,\,\,\,\,}c@{\,\,}|@{\,\,}c@{\,\,\,\,\,}c@{\,\,}|@{\,\,}c@{\,\,\,\,\,}c@{\,\,}|@{\,\,}c@{\,\,\,\,\,}c@{}}
  \hline
$q$ & $t_{2}^{L}$ &
$q$ & $t_{2}^{L}$ &
$q$ & $t_{2}^{L}$ &
$q$ & $t_{2}^{L}$ &
$q$ & $t_{2}^{L}$ &
$q$ & $t_{2}^{L}$ &
$q$ & $t_{2}^{L}$ &
$q$ & $t_{2}^{L^{\vphantom{H^{L}}}}$ \\
 \hline
 7129 & 450  &  7151 & 451  &  7159 & 449  &  7177 & 455  &  7187 & 450  &  7193 & 450  &  7207 & 450  &  7211 & 452  \\
 7213 & 452  &  7219 & 452  &  7229 & 457  &  7237 & 455  &  7243 & 454  &  7247 & 454  &  7253 & 454  &  7283 & 455  \\
 7297 & 455  &  7307 & 457  &  7309 & 454  &  7321 & 453  &  7331 & 457  &  7333 & 459  &  7349 & 460  &  7351 & 458  \\
 7369 & 460  &  7393 & 456  &  7411 & 454  &  7417 & 462  &  7433 & 465  &  7451 & 458  &  7457 & 462  &  7459 & 457  \\
 7477 & 458  &  7481 & 464  &  7487 & 461  &  7489 & 464  &  7499 & 463  &  7507 & 463  &  7517 & 466  &  7523 & 464  \\
 7529 & 468  &  7537 & 462  &  7541 & 460  &  7547 & 464  &  7549 & 465  &  7559 & 465  &  7561 & 462  &  7573 & 464  \\
 7577 & 465  &  7583 & 464  &  7589 & 462  &  7591 & 463  &  7603 & 464  &  7607 & 467  &  7621 & 465  &  7639 & 470  \\
\hline
 \end{tabular}
 }\newpage

%% file: Table4_q_10001_100000.tex
\normalsize{\textbf{Table 4.} The sizes
 $t_{2}^{L}=t_{2}^{L}(2,q)$
 of complete lexiarcs in planes
 $\mathrm{PG}(2,q),$\\ $10000<q<100000$, $q$ 
 prime power}
 \medskip\\ \footnotesize{
 \renewcommand{\arraystretch}{0.92}
 \begin{tabular}
 {@{}c@{\,\,\,\,\,}c@{\,\,}|@{\,\,}c@{\,\,\,\,\,}c@{\,\,}|@{\,\,}c@{\,\,\,\,\,}c@{\,\,}|@{\,\,}c@{\,\,\,\,\,}c@{\,\,}|@{\,\,}c@{\,\,\,\,\,}c@{\,\,}|@{\,\,}c@{\,\,\,\,\,}c@{\,\,}|@{\,\,}c@{\,\,\,\,\,}c@{\,\,}|@{\,\,}c@{\,\,\,\,\,}c@{}}
  \hline
$q$ & $t_{2}^{L}$ &
$q$ & $t_{2}^{L}$ &
$q$ & $t_{2}^{L}$ &
$q$ & $t_{2}^{L}$ &
$q$ & $t_{2}^{L}$ &
$q$ & $t_{2}^{L}$ &
$q$ & $t_{2}^{L}$ &
$q$ & $t_{2}^{L^{\vphantom{H^{L}}}}$ \\
 \hline
 10007 & 544  &  10009 & 543  &  10037 & 542  &  10039 & 541  &  10061 & 543  &  10067 & 544  &  10069 & 544  &  10079 & 546  \\
 10091 & 544  &  10093 & 545  &  10099 & 547  &  10103 & 546  &  10111 & 549  &  10133 & 551  &  10139 & 550  &  10141 & 549  \\
 10151 & 548  &  10159 & 551  &  10163 & 546  &  10169 & 544  &  10177 & 553  &  10181 & 546  &  10193 & 550  &  10201 & 551  \\
 10211 & 550  &  10223 & 551  &  10243 & 546  &  10247 & 552  &  10253 & 545  &  10259 & 551  &  10267 & 559  &  10271 & 547  \\
 10273 & 553  &  10289 & 559  &  10301 & 552  &  10303 & 552  &  10313 & 551  &  10321 & 553  &  10331 & 552  &  10333 & 552  \\
 10337 & 550  &  10343 & 553  &  10357 & 558  &  10369 & 559  &  10391 & 556  &  10399 & 559  &  10427 & 560  &  10429 & 561  \\
 10433 & 557  &  10453 & 561  &  10457 & 560  &  10459 & 556  &  10463 & 561  &  10477 & 557  &  10487 & 563  &  10499 & 560  \\
\hline
 \end{tabular}
} \newpage

\normalsize{\textbf{Table 4.} Continue 1}
\medskip\\\footnotesize{
 \renewcommand{\arraystretch}{0.96}
 \begin{tabular}
 {@{}c@{\,\,\,\,\,}c@{\,\,}|@{\,\,}c@{\,\,\,\,\,}c@{\,\,}|@{\,\,}c@{\,\,\,\,\,}c@{\,\,}|@{\,\,}c@{\,\,\,\,\,}c@{\,\,}|@{\,\,}c@{\,\,\,\,\,}c@{\,\,}|@{\,\,}c@{\,\,\,\,\,}c@{\,\,}|@{\,\,}c@{\,\,\,\,\,}c@{\,\,}|@{\,\,}c@{\,\,\,\,\,}c@{}}
  \hline
$q$ & $t_{2}^{L}$ &
$q$ & $t_{2}^{L}$ &
$q$ & $t_{2}^{L}$ &
$q$ & $t_{2}^{L}$ &
$q$ & $t_{2}^{L}$ &
$q$ & $t_{2}^{L}$ &
$q$ & $t_{2}^{L}$ &
$q$ & $t_{2}^{L^{\vphantom{H^{L}}}}$ \\
 \hline
 13597 & 648  &  13613 & 648  &  13619 & 644  &  13627 & 640  &  13633 & 651  &  13649 & 647  &  13669 & 647  &  13679 & 654  \\
 13681 & 653  &  13687 & 646  &  13691 & 643  &  13693 & 645  &  13697 & 645  &  13709 & 650  &  13711 & 650  &  13721 & 651  \\
 13723 & 648  &  13729 & 650  &  13751 & 649  &  13757 & 651  &  13759 & 656  &  13763 & 647  &  13781 & 647  &  13789 & 653  \\
 13799 & 645  &  13807 & 650  &  13829 & 652  &  13831 & 655  &  13841 & 651  &  13859 & 646  &  13873 & 646  &  13877 & 652  \\
 13879 & 648  &  13883 & 655  &  13901 & 650  &  13903 & 653  &  13907 & 654  &  13913 & 654  &  13921 & 652  &  13931 & 652  \\
 13933 & 652  &  13963 & 651  &  13967 & 651  &  13997 & 656  &  13999 & 657  &  14009 & 657  &  14011 & 655  &  14029 & 657  \\
 14033 & 650  &  14051 & 657  &  14057 & 659  &  14071 & 656  &  14081 & 661  &  14083 & 652  &  14087 & 656  &  14107 & 662  \\
\hline
 \end{tabular}
}\newpage

\normalsize{\textbf{Table 4.} Continue 2 }\medskip\\
\footnotesize{
 \renewcommand{\arraystretch}{0.96}
 \begin{tabular}
 {@{}c@{\,\,\,\,\,}c@{\,\,}|@{\,\,}c@{\,\,\,\,\,}c@{\,\,}|@{\,\,}c@{\,\,\,\,\,}c@{\,\,}|@{\,\,}c@{\,\,\,\,\,}c@{\,\,}|@{\,\,}c@{\,\,\,\,\,}c@{\,\,}|@{\,\,}c@{\,\,\,\,\,}c@{\,\,}|@{\,\,}c@{\,\,\,\,\,}c@{\,\,}|@{\,\,}c@{\,\,\,\,\,}c@{}}
  \hline
$q$ & $t_{2}^{L}$ &
$q$ & $t_{2}^{L}$ &
$q$ & $t_{2}^{L}$ &
$q$ & $t_{2}^{L}$ &
$q$ & $t_{2}^{L}$ &
$q$ & $t_{2}^{L}$ &
$q$ & $t_{2}^{L}$ &
$q$ & $t_{2}^{L^{\vphantom{H^{L}}}}$ \\
 \hline
 17257 & 740  &  17291 & 736  &  17293 & 733  &  17299 & 739  &  17317 & 740  &  17321 & 740  &  17327 & 741  &  17333 & 737  \\
 17341 & 735  &  17351 & 739  &  17359 & 742  &  17377 & 738  &  17383 & 736  &  17387 & 741  &  17389 & 739  &  17393 & 743  \\
 17401 & 739  &  17417 & 737  &  17419 & 750  &  17431 & 742  &  17443 & 740  &  17449 & 735  &  17467 & 739  &  17471 & 743  \\
 17477 & 740  &  17483 & 738  &  17489 & 742  &  17491 & 748  &  17497 & 742  &  17509 & 739  &  17519 & 745  &  17539 & 735  \\
 17551 & 742  &  17569 & 739  &  17573 & 740  &  17579 & 744  &  17581 & 736  &  17597 & 742  &  17599 & 739  &  17609 & 746  \\
 17623 & 747  &  17627 & 744  &  17657 & 748  &  17659 & 751  &  17669 & 750  &  17681 & 745  &  17683 & 749  &  17707 & 749  \\
 17713 & 744  &  17729 & 755  &  17737 & 750  &  17747 & 749  &  17749 & 753  &  17761 & 745  &  17783 & 749  &  17789 & 746  \\
\hline
 \end{tabular}

 }\newpage

\normalsize{\textbf{Table 4.} Continue 3 }\medskip\\
\footnotesize{
 \renewcommand{\arraystretch}{0.96}
 \begin{tabular}
 {@{}c@{\,\,\,\,\,}c@{\,\,}|@{\,\,}c@{\,\,\,\,\,}c@{\,\,}|@{\,\,}c@{\,\,\,\,\,}c@{\,\,}|@{\,\,}c@{\,\,\,\,\,}c@{\,\,}|@{\,\,}c@{\,\,\,\,\,}c@{\,\,}|@{\,\,}c@{\,\,\,\,\,}c@{\,\,}|@{\,\,}c@{\,\,\,\,\,}c@{\,\,}|@{\,\,}c@{\,\,\,\,\,}c@{}}
  \hline
$q$ & $t_{2}^{L}$ &
$q$ & $t_{2}^{L}$ &
$q$ & $t_{2}^{L}$ &
$q$ & $t_{2}^{L}$ &
$q$ & $t_{2}^{L}$ &
$q$ & $t_{2}^{L}$ &
$q$ & $t_{2}^{L}$ &
$q$ & $t_{2}^{L^{\vphantom{H^{L}}}}$ \\
 \hline
 21031 & 817  &  21059 & 825  &  21061 & 820  &  21067 & 826  &  21089 & 825  &  21101 & 823  &  21107 & 828  &  21121 & 824  \\
 21139 & 819  &  21143 & 825  &  21149 & 825  &  21157 & 828  &  21163 & 825  &  21169 & 827  &  21179 & 831  &  21187 & 826  \\
 21191 & 826  &  21193 & 829  &  21211 & 831  &  21221 & 830  &  21227 & 824  &  21247 & 828  &  21269 & 826  &  21277 & 834  \\
 21283 & 830  &  21313 & 826  &  21317 & 832  &  21319 & 829  &  21323 & 828  &  21341 & 829  &  21347 & 830  &  21377 & 827  \\
 21379 & 825  &  21383 & 830  &  21391 & 829  &  21397 & 825  &  21401 & 831  &  21407 & 839  &  21419 & 828  &  21433 & 836  \\
 21467 & 840  &  21481 & 833  &  21487 & 834  &  21491 & 830  &  21493 & 837  &  21499 & 835  &  21503 & 835  &  21517 & 827  \\
 21521 & 833  &  21523 & 836  &  21529 & 831  &  21557 & 844  &  21559 & 833  &  21563 & 831  &  21569 & 828  &  21577 & 834  \\
\hline
 \end{tabular}

 }\newpage

\normalsize{\textbf{Table 4.} Continue 4 }\medskip\\
\footnotesize{
 \renewcommand{\arraystretch}{0.96}
 \begin{tabular}
 {@{}c@{\,\,\,\,\,}c@{\,\,}|@{\,\,}c@{\,\,\,\,\,}c@{\,\,}|@{\,\,}c@{\,\,\,\,\,}c@{\,\,}|@{\,\,}c@{\,\,\,\,\,}c@{\,\,}|@{\,\,}c@{\,\,\,\,\,}c@{\,\,}|@{\,\,}c@{\,\,\,\,\,}c@{\,\,}|@{\,\,}c@{\,\,\,\,\,}c@{\,\,}|@{\,\,}c@{\,\,\,\,\,}c@{}}
  \hline
$q$ & $t_{2}^{L}$ &
$q$ & $t_{2}^{L}$ &
$q$ & $t_{2}^{L}$ &
$q$ & $t_{2}^{L}$ &
$q$ & $t_{2}^{L}$ &
$q$ & $t_{2}^{L}$ &
$q$ & $t_{2}^{L}$ &
$q$ & $t_{2}^{L^{\vphantom{H^{L}}}}$ \\
 \hline
 24847 & 899  &  24851 & 899  &  24859 & 901  &  24877 & 901  &  24889 & 907  &  24907 & 901  &  24917 & 905  &  24919 & 896  \\
 24923 & 904  &  24943 & 901  &  24953 & 902  &  24967 & 905  &  24971 & 900  &  24977 & 908  &  24979 & 904  &  24989 & 901  \\
 25013 & 905  &  25031 & 906  &  25033 & 904  &  25037 & 912  &  25057 & 902  &  25073 & 903  &  25087 & 906  &  25097 & 916  \\
 25111 & 904  &  25117 & 907  &  25121 & 904  &  25127 & 906  &  25147 & 909  &  25153 & 911  &  25163 & 909  &  25169 & 913  \\
 25171 & 904  &  25183 & 908  &  25189 & 907  &  25219 & 910  &  25229 & 911  &  25237 & 911  &  25243 & 910  &  25247 & 907  \\
 25253 & 911  &  25261 & 906  &  25301 & 905  &  25303 & 912  &  25307 & 908  &  25309 & 909  &  25321 & 906  &  25339 & 907  \\
 25343 & 908  &  25349 & 910  &  25357 & 913  &  25367 & 909  &  25373 & 914  &  25391 & 909  &  25409 & 912  &  25411 & 907  \\
\hline
 \end{tabular}

 }\newpage

\normalsize{\textbf{Table 4.} Continue 5 }\medskip\\
\footnotesize{
 \renewcommand{\arraystretch}{0.96}
 \begin{tabular}
 {@{}c@{\,\,\,\,\,}r@{\,\,}|@{\,\,}c@{\,\,\,\,\,}r@{\,\,}|@{\,\,}c@{\,\,\,\,\,}r@{\,}|@{\,\,}c@{\,\,\,\,\,}r@{\,}|@{\,\,}c@{\,\,\,\,\,}r@{\,}|@{\,\,}c@{\,\,\,\,\,}r@{\,\,}|@{\,\,}c@{\,\,\,\,\,}r@{\,\,}|@{\,\,}c@{\,\,\,\,}r@{}}
  \hline
$q$ & $t_{2}^{L}$ &
$q$ & $t_{2}^{L}$ &
$q$ & $t_{2}^{L}$ &
$q$ & $t_{2}^{L}$ &
$q$ & $t_{2}^{L}$ &
$q$ & $t_{2}^{L}$ &
$q$ & $t_{2}^{L}$ &
$q$ & $t_{2}^{L^{\vphantom{H^{L}}}}$ \\
 \hline
 28711 & 974  &  28723 & 972  &  28729 & 973  &  28751 & 981  &  28753 & 981  &  28759 & 980  &  28771 & 976  &  28789 & 984  \\
 28793 & 986  &  28807 & 976  &  28813 & 982  &  28817 & 982  &  28837 & 973  &  28843 & 977  &  28859 & 985  &  28867 & 976  \\
 28871 & 976  &  28879 & 982  &  28901 & 980  &  28909 & 983  &  28921 & 980  &  28927 & 982  &  28933 & 981  &  28949 & 979  \\
 28961 & 986  &  28979 & 979  &  29009 & 981  &  29017 & 982  &  29021 & 988  &  29023 & 979  &  29027 & 975  &  29033 & 982  \\
 29059 & 984  &  29063 & 979  &  29077 & 983  &  29101 & 984  &  29123 & 987  &  29129 & 985  &  29131 & 978  &  29137 & 989  \\
 29147 & 985  &  29153 & 986  &  29167 & 980  &  29173 & 987  &  29179 & 981  &  29191 & 987  &  29201 & 991  &  29207 & 984  \\
 29209 & 987  &  29221 & 986  &  29231 & 987  &  29243 & 980  &  29251 & 990  &  29269 & 987  &  29287 & 994  &  29297 & 991  \\
\hline
 \end{tabular}

 }\newpage

\normalsize{\textbf{Table 4.} Continue 6 }\medskip\\
\footnotesize{
 \renewcommand{\arraystretch}{0.96}
 \begin{tabular}
 {@{}c@{\,\,\,\,\,}c@{\,\,}|@{\,\,}c@{\,\,\,\,\,}c@{\,\,}|@{\,\,}c@{\,\,\,\,\,}c@{\,\,}|@{\,\,}c@{\,\,\,\,\,}c@{\,\,}|@{\,\,}c@{\,\,\,\,\,}c@{\,}|@{\,\,}c@{\,\,\,\,\,}c@{\,}|@{\,\,}c@{\,\,\,\,\,}c@{\,}|@{\,\,}c@{\,\,\,\,}c@{}}
  \hline
$q$ & $t_{2}^{L}$ &
$q$ & $t_{2}^{L}$ &
$q$ & $t_{2}^{L}$ &
$q$ & $t_{2}^{L}$ &
$q$ & $t_{2}^{L}$ &
$q$ & $t_{2}^{L}$ &
$q$ & $t_{2}^{L}$ &
$q$ & $t_{2}^{L^{\vphantom{H^{L}}}}$ \\
 \hline
 32713 & 1055  &  32717 & 1049  &  32719 & 1048  &  32749 & 1046  &  32761 & 1055  &  32768 & 1048  &  32771 & 1045  &  32779 & 1035  \\
 32783 & 1046  &  32789 & 1053  &  32797 & 1050  &  32801 & 1050  &  32803 & 1055  &  32831 & 1053  &  32833 & 1048  &  32839 & 1056  \\
 32843 & 1054  &  32869 & 1054  &  32887 & 1051  &  32909 & 1052  &  32911 & 1056  &  32917 & 1052  &  32933 & 1056  &  32939 & 1055  \\
 32941 & 1059  &  32957 & 1052  &  32969 & 1059  &  32971 & 1055  &  32983 & 1061  &  32987 & 1050  &  32993 & 1054  &  32999 & 1052  \\
 33013 & 1061  &  33023 & 1058  &  33029 & 1053  &  33037 & 1056  &  33049 & 1056  &  33053 & 1058  &  33071 & 1059  &  33073 & 1054  \\
 33083 & 1056  &  33091 & 1056  &  33107 & 1060  &  33113 & 1051  &  33119 & 1063  &  33149 & 1060  &  33151 & 1060  &  33161 & 1066  \\
 33179 & 1063  &  33181 & 1054  &  33191 & 1059  &  33199 & 1054  &  33203 & 1059  &  33211 & 1059  &  33223 & 1055  &  33247 & 1054  \\
\hline
 \end{tabular}

 }\newpage

\normalsize{\textbf{Table 4.} Continue 7 }\medskip\\
\footnotesize{
 \renewcommand{\arraystretch}{0.96}
 \begin{tabular}
 {@{}c@{\,\,\,\,\,}c@{\,\,}|@{\,\,}c@{\,\,\,\,\,}c@{\,\,}|@{\,\,}c@{\,\,\,\,\,}c@{\,\,}|@{\,\,}c@{\,\,\,\,\,}c@{\,\,}|@{\,\,}c@{\,\,\,\,\,}c@{\,}|@{\,\,}c@{\,\,\,\,\,}c@{\,}|@{\,\,}c@{\,\,\,\,\,}c@{\,}|@{\,\,}c@{\,\,\,\,}c@{}}
  \hline
$q$ & $t_{2}^{L}$ &
$q$ & $t_{2}^{L}$ &
$q$ & $t_{2}^{L}$ &
$q$ & $t_{2}^{L}$ &
$q$ & $t_{2}^{L}$ &
$q$ & $t_{2}^{L}$ &
$q$ & $t_{2}^{L}$ &
$q$ & $t_{2}^{L^{\vphantom{H^{L}}}}$ \\
 \hline
 36691 & 1118  &  36697 & 1120  &  36709 & 1116  &  36713 & 1118  &  36721 & 1128  &  36739 & 1121  &  36749 & 1115  &  36761 & 1120  \\
 36767 & 1116  &  36779 & 1126  &  36781 & 1118  &  36787 & 1122  &  36791 & 1126  &  36793 & 1121  &  36809 & 1131  &  36821 & 1127  \\
 36833 & 1122  &  36847 & 1125  &  36857 & 1122  &  36871 & 1111  &  36877 & 1123  &  36887 & 1127  &  36899 & 1123  &  36901 & 1120  \\
 36913 & 1122  &  36919 & 1123  &  36923 & 1117  &  36929 & 1125  &  36931 & 1127  &  36943 & 1127  &  36947 & 1114  &  36973 & 1122  \\
 36979 & 1121  &  36997 & 1128  &  37003 & 1121  &  37013 & 1123  &  37019 & 1127  &  37021 & 1125  &  37039 & 1121  &  37049 & 1117  \\
 37057 & 1125  &  37061 & 1138  &  37087 & 1125  &  37097 & 1124  &  37117 & 1121  &  37123 & 1127  &  37139 & 1123  &  37159 & 1126  \\
 37171 & 1127  &  37181 & 1128  &  37189 & 1121  &  37199 & 1132  &  37201 & 1127  &  37217 & 1127  &  37223 & 1127  &  37243 & 1136  \\
\hline
 \end{tabular}

 }\newpage

\normalsize{\textbf{Table 4.} Continue 8 }\medskip\\
\footnotesize{
 \renewcommand{\arraystretch}{0.96}
 \begin{tabular}
 {@{}c@{\,\,\,\,\,}c@{\,\,}|@{\,\,}c@{\,\,\,\,\,}c@{\,\,}|@{\,\,}c@{\,\,\,\,\,}c@{\,\,}|@{\,\,}c@{\,\,\,\,\,}c@{\,\,}|@{\,\,}c@{\,\,\,\,\,}c@{\,}|@{\,\,}c@{\,\,\,\,\,}c@{\,}|@{\,\,}c@{\,\,\,\,\,}c@{\,}|@{\,\,}c@{\,\,\,\,}c@{}}
  \hline
$q$ & $t_{2}^{L}$ &
$q$ & $t_{2}^{L}$ &
$q$ & $t_{2}^{L}$ &
$q$ & $t_{2}^{L}$ &
$q$ & $t_{2}^{L}$ &
$q$ & $t_{2}^{L}$ &
$q$ & $t_{2}^{L}$ &
$q$ & $t_{2}^{L^{\vphantom{H^{L}}}}$ \\
 \hline
 40801 & 1187  &  40813 & 1183  &  40819 & 1184  &  40823 & 1186  &  40829 & 1192  &  40841 & 1188  &  40847 & 1194  &  40849 & 1187  \\
 40853 & 1184  &  40867 & 1187  &  40879 & 1186  &  40883 & 1187  &  40897 & 1187  &  40903 & 1191  &  40927 & 1188  &  40933 & 1193  \\
 40939 & 1185  &  40949 & 1191  &  40961 & 1183  &  40973 & 1183  &  40993 & 1188  &  41011 & 1190  &  41017 & 1184  &  41023 & 1192  \\
 41039 & 1187  &  41047 & 1192  &  41051 & 1186  &  41057 & 1189  &  41077 & 1190  &  41081 & 1199  &  41113 & 1195  &  41117 & 1190  \\
 41131 & 1190  &  41141 & 1190  &  41143 & 1199  &  41149 & 1193  &  41161 & 1190  &  41177 & 1191  &  41179 & 1188  &  41183 & 1186  \\
 41189 & 1185  &  41201 & 1193  &  41203 & 1195  &  41213 & 1183  &  41221 & 1196  &  41227 & 1197  &  41231 & 1189  &  41233 & 1188  \\
 41243 & 1193  &  41257 & 1199  &  41263 & 1195  &  41269 & 1192  &  41281 & 1191  &  41299 & 1195  &  41333 & 1193  &  41341 & 1194  \\
\hline
 \end{tabular}

 }\newpage

\normalsize{\textbf{Table 4.} Continue 9 }\medskip\\
\footnotesize{
 \renewcommand{\arraystretch}{0.96}
 \begin{tabular}
 {@{}c@{\,\,\,\,\,}c@{\,\,}|@{\,\,}c@{\,\,\,\,\,}c@{\,\,}|@{\,\,}c@{\,\,\,\,\,}c@{\,\,}|@{\,\,}c@{\,\,\,\,\,}c@{\,\,}|@{\,\,}c@{\,\,\,\,\,}c@{\,}|@{\,\,}c@{\,\,\,\,\,}c@{\,}|@{\,\,}c@{\,\,\,\,\,}c@{\,}|@{\,\,}c@{\,\,\,\,}c@{}}
  \hline
$q$ & $t_{2}^{L}$ &
$q$ & $t_{2}^{L}$ &
$q$ & $t_{2}^{L}$ &
$q$ & $t_{2}^{L}$ &
$q$ & $t_{2}^{L}$ &
$q$ & $t_{2}^{L}$ &
$q$ & $t_{2}^{L}$ &
$q$ & $t_{2}^{L^{\vphantom{H^{L}}}}$ \\
 \hline
 44777 & 1253  &  44789 & 1243  &  44797 & 1250  &  44809 & 1245  &  44819 & 1256  &  44839 & 1247  &  44843 & 1237  &  44851 & 1245  \\
 44867 & 1249  &  44879 & 1255  &  44887 & 1253  &  44893 & 1255  &  44909 & 1245  &  44917 & 1250  &  44927 & 1249  &  44939 & 1256  \\
 44953 & 1251  &  44959 & 1246  &  44963 & 1255  &  44971 & 1242  &  44983 & 1248  &  44987 & 1251  &  45007 & 1250  &  45013 & 1253  \\
 45053 & 1251  &  45061 & 1247  &  45077 & 1251  &  45083 & 1259  &  45119 & 1254  &  45121 & 1254  &  45127 & 1245  &  45131 & 1256  \\
 45137 & 1251  &  45139 & 1250  &  45161 & 1264  &  45179 & 1254  &  45181 & 1256  &  45191 & 1258  &  45197 & 1258  &  45233 & 1258  \\
 45247 & 1252  &  45259 & 1267  &  45263 & 1256  &  45281 & 1262  &  45289 & 1259  &  45293 & 1263  &  45307 & 1255  &  45317 & 1264  \\
 45319 & 1254  &  45329 & 1255  &  45337 & 1254  &  45341 & 1254  &  45343 & 1254  &  45361 & 1252  &  45377 & 1254  &  45389 & 1250  \\
\hline
 \end{tabular}

 }\newpage

\normalsize{\textbf{Table 4.} Continue 10 }\medskip\\
\footnotesize{
 \renewcommand{\arraystretch}{0.96}
 \begin{tabular}
 {@{}c@{\,\,\,\,\,}c@{\,\,}|@{\,\,}c@{\,\,\,\,\,}c@{\,\,}|@{\,\,}c@{\,\,\,\,\,}c@{\,\,}|@{\,\,}c@{\,\,\,\,\,}c@{\,\,}|@{\,\,}c@{\,\,\,\,\,}c@{\,}|@{\,\,}c@{\,\,\,\,\,}c@{\,}|@{\,\,}c@{\,\,\,\,\,}c@{\,\,}|@{\,}c@{\,\,\,\,}c@{}}
  \hline
$q$ & $t_{2}^{L}$ &
$q$ & $t_{2}^{L}$ &
$q$ & $t_{2}^{L}$ &
$q$ & $t_{2}^{L}$ &
$q$ & $t_{2}^{L}$ &
$q$ & $t_{2}^{L}$ &
$q$ & $t_{2}^{L}$ &
$q$ & $t_{2}^{L^{\vphantom{H^{L}}}}$ \\
 \hline
 49019 & 1313  &  49031 & 1314  &  49033 & 1316  &  49037 & 1314  &  49043 & 1314  &  49057 & 1312  &  49069 & 1313  &  49081 & 1310  \\
 49103 & 1308  &  49109 & 1308  &  49117 & 1311  &  49121 & 1308  &  49123 & 1310  &  49139 & 1310  &  49157 & 1322  &  49169 & 1310  \\
 49171 & 1314  &  49177 & 1309  &  49193 & 1318  &  49199 & 1317  &  49201 & 1317  &  49207 & 1309  &  49211 & 1311  &  49223 & 1311  \\
 49253 & 1317  &  49261 & 1310  &  49277 & 1313  &  49279 & 1310  &  49297 & 1312  &  49307 & 1314  &  49331 & 1316  &  49333 & 1317  \\
 49339 & 1315  &  49363 & 1318  &  49367 & 1315  &  49369 & 1322  &  49391 & 1322  &  49393 & 1330  &  49409 & 1320  &  49411 & 1316  \\
 49417 & 1314  &  49429 & 1311  &  49433 & 1323  &  49451 & 1315  &  49459 & 1323  &  49463 & 1312  &  49477 & 1320  &  49481 & 1325  \\
 49499 & 1314  &  49523 & 1313  &  49529 & 1315  &  49531 & 1322  &  49537 & 1313  &  49547 & 1319  &  49549 & 1310  &  49559 & 1320  \\
\hline
 \end{tabular}

 }\newpage

\normalsize{\textbf{Table 4.} Continue 11 }\medskip\\
\footnotesize{
 \renewcommand{\arraystretch}{0.96}
 \begin{tabular}
 {@{}c@{\,\,\,\,\,}c@{\,\,}|@{\,\,}c@{\,\,\,\,\,}c@{\,\,}|@{\,\,}c@{\,\,\,\,\,}c@{\,\,}|@{\,\,}c@{\,\,\,\,\,}c@{\,\,}|@{\,\,}c@{\,\,\,\,\,}c@{\,}|@{\,\,}c@{\,\,\,\,\,}c@{\,}|@{\,\,}c@{\,\,\,\,\,}c@{\,}|@{\,\,}c@{\,\,\,\,}c@{}}
  \hline
$q$ & $t_{2}^{L}$ &
$q$ & $t_{2}^{L}$ &
$q$ & $t_{2}^{L}$ &
$q$ & $t_{2}^{L}$ &
$q$ & $t_{2}^{L}$ &
$q$ & $t_{2}^{L}$ &
$q$ & $t_{2}^{L}$ &
$q$ & $t_{2}^{L^{\vphantom{H^{L}}}}$ \\
 \hline
 53101 & 1370  &  53113 & 1368  &  53117 & 1376  &  53129 & 1378  &  53147 & 1367  &  53149 & 1373  &  53161 & 1370  &  53171 & 1368  \\
 53173 & 1368  &  53189 & 1373  &  53197 & 1369  &  53201 & 1367  &  53231 & 1375  &  53233 & 1377  &  53239 & 1364  &  53267 & 1369  \\
 53269 & 1368  &  53279 & 1377  &  53281 & 1377  &  53299 & 1368  &  53309 & 1366  &  53323 & 1373  &  53327 & 1377  &  53353 & 1377  \\
 53359 & 1378  &  53377 & 1375  &  53381 & 1377  &  53401 & 1371  &  53407 & 1371  &  53411 & 1379  &  53419 & 1376  &  53437 & 1367  \\
 53441 & 1376  &  53453 & 1379  &  53479 & 1380  &  53503 & 1371  &  53507 & 1369  &  53527 & 1377  &  53549 & 1382  &  53551 & 1376  \\
 53569 & 1375  &  53591 & 1375  &  53593 & 1380  &  53597 & 1376  &  53609 & 1365  &  53611 & 1368  &  53617 & 1382  &  53623 & 1372  \\
 53629 & 1378  &  53633 & 1388  &  53639 & 1372  &  53653 & 1376  &  53657 & 1375  &  53681 & 1377  &  53693 & 1379  &  53699 & 1387  \\
\hline
 \end{tabular}

 }\newpage

\normalsize{\textbf{Table 4.} Continue 12 }\medskip\\
\footnotesize{
 \renewcommand{\arraystretch}{0.96}
 \begin{tabular}
 {@{}c@{\,\,\,\,\,}c@{\,\,}|@{\,\,}c@{\,\,\,\,\,}c@{\,\,}|@{\,\,}c@{\,\,\,\,\,}c@{\,\,}|@{\,\,}c@{\,\,\,\,\,}c@{\,}|@{\,\,}c@{\,\,\,\,\,}c@{\,}|@{\,\,}c@{\,\,\,\,\,}c@{\,}|@{\,\,}c@{\,\,\,\,\,}c@{\,}|@{\,\,}c@{\,\,\,\,}c@{}}
  \hline
$q$ & $t_{2}^{L}$ &
$q$ & $t_{2}^{L}$ &
$q$ & $t_{2}^{L}$ &
$q$ & $t_{2}^{L}$ &
$q$ & $t_{2}^{L}$ &
$q$ & $t_{2}^{L}$ &
$q$ & $t_{2}^{L}$ &
$q$ & $t_{2}^{L^{\vphantom{H^{L}}}}$ \\
 \hline
 57191 & 1435  &  57193 & 1426  &  57203 & 1438  &  57221 & 1420  &  57223 & 1440  &  57241 & 1422  &  57251 & 1420  &  57259 & 1427  \\
 57269 & 1428  &  57271 & 1424  &  57283 & 1426  &  57287 & 1425  &  57301 & 1418  &  57329 & 1432  &  57331 & 1431  &  57347 & 1429  \\
 57349 & 1430  &  57367 & 1426  &  57373 & 1420  &  57383 & 1426  &  57389 & 1432  &  57397 & 1426  &  57413 & 1432  &  57427 & 1426  \\
 57457 & 1431  &  57467 & 1426  &  57487 & 1429  &  57493 & 1434  &  57503 & 1427  &  57527 & 1428  &  57529 & 1427  &  57557 & 1434  \\
 57559 & 1428  &  57571 & 1438  &  57587 & 1430  &  57593 & 1435  &  57601 & 1432  &  57637 & 1433  &  57641 & 1433  &  57649 & 1433  \\
 57653 & 1434  &  57667 & 1435  &  57679 & 1430  &  57689 & 1427  &  57697 & 1429  &  57709 & 1430  &  57713 & 1430  &  57719 & 1439  \\
 57727 & 1434  &  57731 & 1434  &  57737 & 1437  &  57751 & 1430  &  57773 & 1438  &  57781 & 1436  &  57787 & 1431  &  57791 & 1428  \\
\hline
 \end{tabular}

 }\newpage

\normalsize{\textbf{Table 4.} Continue 13 }\medskip\\
\footnotesize{
 \renewcommand{\arraystretch}{0.96}
 \begin{tabular}
 {@{}c@{\,\,\,\,\,}c@{\,\,}|@{\,\,}c@{\,\,\,\,\,}c@{\,\,}|@{\,\,}c@{\,\,\,\,\,}c@{\,\,}|@{\,\,}c@{\,\,\,\,\,}c@{\,}|@{\,\,}c@{\,\,\,\,\,}c@{\,}|@{\,\,}c@{\,\,\,\,\,}c@{\,}|@{\,\,}c@{\,\,\,\,\,}c@{\,}|@{\,\,}c@{\,\,\,\,}c@{}}
  \hline
$q$ & $t_{2}^{L}$ &
$q$ & $t_{2}^{L}$ &
$q$ & $t_{2}^{L}$ &
$q$ & $t_{2}^{L}$ &
$q$ & $t_{2}^{L}$ &
$q$ & $t_{2}^{L}$ &
$q$ & $t_{2}^{L}$ &
$q$ & $t_{2}^{L^{\vphantom{H^{L}}}}$ \\
 \hline
 61469 & 1487  &  61471 & 1482  &  61483 & 1489  &  61487 & 1488  &  61493 & 1484  &  61507 & 1481  &  61511 & 1485  &  61519 & 1483  \\
 61543 & 1481  &  61547 & 1490  &  61553 & 1481  &  61559 & 1478  &  61561 & 1484  &  61583 & 1484  &  61603 & 1483  &  61609 & 1488  \\
 61613 & 1489  &  61627 & 1487  &  61631 & 1490  &  61637 & 1485  &  61643 & 1488  &  61651 & 1483  &  61657 & 1482  &  61667 & 1489  \\
 61673 & 1483  &  61681 & 1490  &  61687 & 1487  &  61703 & 1491  &  61717 & 1479  &  61723 & 1486  &  61729 & 1488  &  61751 & 1496  \\
 61757 & 1493  &  61781 & 1487  &  61813 & 1486  &  61819 & 1491  &  61837 & 1489  &  61843 & 1492  &  61861 & 1487  &  61871 & 1486  \\
 61879 & 1485  &  61909 & 1497  &  61927 & 1495  &  61933 & 1488  &  61949 & 1484  &  61961 & 1487  &  61967 & 1485  &  61979 & 1500  \\
 61981 & 1491  &  61987 & 1489  &  61991 & 1484  &  62003 & 1490  &  62011 & 1491  &  62017 & 1500  &  62039 & 1494  &  62047 & 1489  \\
\hline
 \end{tabular}

 }\newpage

\normalsize{\textbf{Table 4.} Continue 14 }\medskip\\
\footnotesize{
 \renewcommand{\arraystretch}{0.96}
 \begin{tabular}
 {@{}c@{\,\,\,\,\,}c@{\,\,}|@{\,\,}c@{\,\,\,\,\,}c@{\,\,}|@{\,\,}c@{\,\,\,\,\,}c@{\,\,}|@{\,\,}c@{\,\,\,\,\,}c@{\,\,}|@{\,\,}c@{\,\,\,\,}c@{\,\,}|@{\,\,}c@{\,\,\,\,}c@{\,\,}|@{\,\,}c@{\,\,\,\,}c@{\,\,}|@{\,\,}c@{\,\,\,\,}c@{}}
  \hline
$q$ & $t_{2}^{L}$ &
$q$ & $t_{2}^{L}$ &
$q$ & $t_{2}^{L}$ &
$q$ & $t_{2}^{L}$ &
$q$ & $t_{2}^{L}$ &
$q$ & $t_{2}^{L}$ &
$q$ & $t_{2}^{L}$ &
$q$ & $t_{2}^{L^{\vphantom{H^{L}}}}$ \\
 \hline
 65707 & 1535  &  65713 & 1538  &  65717 & 1545  &  65719 & 1541  &  65729 & 1536  &  65731 & 1538  &  65761 & 1552  &  65777 & 1534  \\
 65789 & 1541  &  65809 & 1536  &  65827 & 1531  &  65831 & 1544  &  65837 & 1537  &  65839 & 1544  &  65843 & 1550  &  65851 & 1538  \\
 65867 & 1540  &  65881 & 1543  &  65899 & 1547  &  65921 & 1541  &  65927 & 1537  &  65929 & 1536  &  65951 & 1538  &  65957 & 1540  \\
 65963 & 1542  &  65981 & 1546  &  65983 & 1544  &  65993 & 1542  &  66029 & 1543  &  66037 & 1542  &  66041 & 1537  &  66047 & 1539  \\
 66049 & 1542  &  66067 & 1552  &  66071 & 1545  &  66083 & 1545  &  66089 & 1535  &  66103 & 1536  &  66107 & 1544  &  66109 & 1551  \\
 66137 & 1549  &  66161 & 1538  &  66169 & 1542  &  66173 & 1549  &  66179 & 1543  &  66191 & 1548  &  66221 & 1551  &  66239 & 1542  \\
 66271 & 1540  &  66293 & 1558  &  66301 & 1548  &  66337 & 1540  &  66343 & 1544  &  66347 & 1546  &  66359 & 1549  &  66361 & 1548  \\
\hline
 \end{tabular}

 }\newpage

\normalsize{\textbf{Table 4.} Continue 15 }\medskip\\
\footnotesize{
 \renewcommand{\arraystretch}{0.96}
 \begin{tabular}
 {@{}c@{\,\,\,\,\,}c@{\,\,}|@{\,\,}c@{\,\,\,\,\,}c@{\,\,}|@{\,\,}c@{\,\,\,\,\,}c@{\,\,}|@{\,\,}c@{\,\,\,\,\,}c@{\,\,}|@{\,\,}c@{\,\,\,\,\,}c@{\,\,}|@{\,\,}c@{\,\,\,\,}c@{\,\,}|@{\,\,}c@{\,\,\,\,}c@{\,\,}|@{\,}c@{\,\,\,\,}c@{}}
  \hline
$q$ & $t_{2}^{L}$ &
$q$ & $t_{2}^{L}$ &
$q$ & $t_{2}^{L}$ &
$q$ & $t_{2}^{L}$ &
$q$ & $t_{2}^{L}$ &
$q$ & $t_{2}^{L}$ &
$q$ & $t_{2}^{L}$ &
$q$ & $t_{2}^{L^{\vphantom{H^{L}}}}$ \\
 \hline
 70099 & 1595  &  70111 & 1586  &  70117 & 1596  &  70121 & 1595  &  70123 & 1599  &  70139 & 1591  &  70141 & 1598  &  70157 & 1588  \\
 70163 & 1597  &  70177 & 1586  &  70181 & 1591  &  70183 & 1607  &  70199 & 1595  &  70201 & 1589  &  70207 & 1589  &  70223 & 1596  \\
 70229 & 1588  &  70237 & 1601  &  70241 & 1596  &  70249 & 1597  &  70271 & 1601  &  70289 & 1596  &  70297 & 1595  &  70309 & 1597  \\
 70313 & 1595  &  70321 & 1599  &  70327 & 1593  &  70351 & 1607  &  70373 & 1593  &  70379 & 1601  &  70381 & 1601  &  70393 & 1599  \\
 70423 & 1605  &  70429 & 1600  &  70439 & 1599  &  70451 & 1582  &  70457 & 1601  &  70459 & 1598  &  70481 & 1602  &  70487 & 1599  \\
 70489 & 1598  &  70501 & 1598  &  70507 & 1604  &  70529 & 1600  &  70537 & 1595  &  70549 & 1604  &  70571 & 1596  &  70573 & 1602  \\
 70583 & 1601  &  70589 & 1591  &  70607 & 1599  &  70619 & 1601  &  70621 & 1593  &  70627 & 1597  &  70639 & 1598  &  70657 & 1603  \\
\hline
 \end{tabular}

 }\newpage

\normalsize{\textbf{Table 4.} Continue 16 }\medskip\\
\footnotesize{
 \renewcommand{\arraystretch}{0.96}
 \begin{tabular}
 {@{}c@{\,\,\,\,\,}c@{\,\,}|@{\,\,}c@{\,\,\,\,\,}c@{\,\,}|@{\,\,}c@{\,\,\,\,\,}c@{\,\,}|@{\,\,}c@{\,\,\,\,\,}c@{\,\,}|@{\,\,}c@{\,\,\,\,}c@{\,\,}|@{\,\,}c@{\,\,\,\,}c@{\,\,}|@{\,\,}c@{\,\,\,\,}c@{\,\,}|@{\,\,}c@{\,\,\,\,}c@{}}
  \hline
$q$ & $t_{2}^{L}$ &
$q$ & $t_{2}^{L}$ &
$q$ & $t_{2}^{L}$ &
$q$ & $t_{2}^{L}$ &
$q$ & $t_{2}^{L}$ &
$q$ & $t_{2}^{L}$ &
$q$ & $t_{2}^{L}$ &
$q$ & $t_{2}^{L^{\vphantom{H^{L}}}}$ \\
 \hline
 74279 & 1648  &  74287 & 1652  &  74293 & 1648  &  74297 & 1639  &  74311 & 1646  &  74317 & 1648  &  74323 & 1648  &  74353 & 1642  \\
 74357 & 1645  &  74363 & 1648  &  74377 & 1648  &  74381 & 1650  &  74383 & 1648  &  74411 & 1648  &  74413 & 1644  &  74419 & 1646  \\
 74441 & 1656  &  74449 & 1643  &  74453 & 1648  &  74471 & 1648  &  74489 & 1653  &  74507 & 1646  &  74509 & 1641  &  74521 & 1639  \\
 74527 & 1642  &  74531 & 1649  &  74551 & 1642  &  74561 & 1653  &  74567 & 1651  &  74573 & 1644  &  74587 & 1652  &  74597 & 1652  \\
 74609 & 1645  &  74611 & 1652  &  74623 & 1652  &  74653 & 1654  &  74687 & 1650  &  74699 & 1647  &  74707 & 1656  &  74713 & 1646  \\
 74717 & 1650  &  74719 & 1646  &  74729 & 1646  &  74731 & 1657  &  74747 & 1648  &  74759 & 1651  &  74761 & 1654  &  74771 & 1647  \\
 74779 & 1654  &  74797 & 1645  &  74821 & 1649  &  74827 & 1658  &  74831 & 1659  &  74843 & 1648  &  74857 & 1656  &  74861 & 1657  \\
\hline
 \end{tabular}

 }\newpage

\normalsize{\textbf{Table 4.} Continue 17 }\medskip\\
\footnotesize{
 \renewcommand{\arraystretch}{0.96}
 \begin{tabular}
 {@{}c@{\,\,\,\,\,}c@{\,\,}|@{\,\,}c@{\,\,\,\,\,}c@{\,\,}|@{\,\,}c@{\,\,\,\,\,}c@{\,\,}|@{\,\,}c@{\,\,\,\,\,}c@{\,\,}|@{\,\,}c@{\,\,\,\,}c@{\,\,}|@{\,\,}c@{\,\,\,\,}c@{\,\,}|@{\,\,}c@{\,\,\,\,}c@{\,\,}|@{\,\,}c@{\,\,\,\,}c@{}}
  \hline
$q$ & $t_{2}^{L}$ &
$q$ & $t_{2}^{L}$ &
$q$ & $t_{2}^{L}$ &
$q$ & $t_{2}^{L}$ &
$q$ & $t_{2}^{L}$ &
$q$ & $t_{2}^{L}$ &
$q$ & $t_{2}^{L}$ &
$q$ & $t_{2}^{L^{\vphantom{H^{L}}}}$ \\
 \hline
 78569 & 1697  &  78571 & 1688  &  78577 & 1697  &  78583 & 1707  &  78593 & 1697  &  78607 & 1689  &  78623 & 1693  &  78643 & 1694  \\
 78649 & 1696  &  78653 & 1701  &  78691 & 1699  &  78697 & 1706  &  78707 & 1697  &  78713 & 1704  &  78721 & 1700  &  78737 & 1701  \\
 78779 & 1699  &  78781 & 1701  &  78787 & 1704  &  78791 & 1697  &  78797 & 1712  &  78803 & 1701  &  78809 & 1713  &  78823 & 1696  \\
 78839 & 1702  &  78853 & 1697  &  78857 & 1703  &  78877 & 1704  &  78887 & 1714  &  78889 & 1699  &  78893 & 1703  &  78901 & 1707  \\
 78919 & 1700  &  78929 & 1700  &  78941 & 1699  &  78961 & 1704  &  78977 & 1696  &  78979 & 1697  &  78989 & 1702  &  79031 & 1696  \\
 79039 & 1710  &  79043 & 1703  &  79063 & 1704  &  79087 & 1699  &  79103 & 1701  &  79111 & 1702  &  79133 & 1698  &  79139 & 1695  \\
 79147 & 1705  &  79151 & 1701  &  79153 & 1702  &  79159 & 1702  &  79181 & 1710  &  79187 & 1700  &  79193 & 1701  &  79201 & 1700  \\
\hline
 \end{tabular}

 }\newpage

\normalsize{\textbf{Table 4.} Continue 18 }\medskip\\
\footnotesize{
 \renewcommand{\arraystretch}{0.96}
 \begin{tabular}
 {@{}c@{\,\,\,\,\,}c@{\,\,}|@{\,\,}c@{\,\,\,\,\,}c@{\,\,}|@{\,\,}c@{\,\,\,\,\,}c@{\,\,}|@{\,\,}c@{\,\,\,\,\,}c@{\,\,}|@{\,\,}c@{\,\,\,\,}c@{\,\,}|@{\,\,}c@{\,\,\,\,}c@{\,\,}|@{\,\,}c@{\,\,\,\,}c@{\,\,}|@{\,\,}c@{\,\,\,\,}c@{}}
  \hline
$q$ & $t_{2}^{L}$ &
$q$ & $t_{2}^{L}$ &
$q$ & $t_{2}^{L}$ &
$q$ & $t_{2}^{L}$ &
$q$ & $t_{2}^{L}$ &
$q$ & $t_{2}^{L}$ &
$q$ & $t_{2}^{L}$ &
$q$ & $t_{2}^{L^{\vphantom{H^{L}}}}$ \\
 \hline
 82781 & 1753  &  82787 & 1750  &  82793 & 1750  &  82799 & 1740  &  82811 & 1745  &  82813 & 1734  &  82837 & 1744  &  82847 & 1751  \\
 82883 & 1738  &  82889 & 1745  &  82891 & 1751  &  82903 & 1744  &  82913 & 1744  &  82939 & 1751  &  82963 & 1746  &  82981 & 1758  \\
 82997 & 1749  &  83003 & 1740  &  83009 & 1746  &  83023 & 1749  &  83047 & 1748  &  83059 & 1748  &  83063 & 1752  &  83071 & 1755  \\
 83077 & 1751  &  83089 & 1751  &  83093 & 1754  &  83101 & 1747  &  83117 & 1752  &  83137 & 1749  &  83177 & 1747  &  83203 & 1750  \\
 83207 & 1748  &  83219 & 1754  &  83221 & 1752  &  83227 & 1755  &  83231 & 1755  &  83233 & 1754  &  83243 & 1746  &  83257 & 1756  \\
 83267 & 1755  &  83269 & 1746  &  83273 & 1751  &  83299 & 1758  &  83311 & 1751  &  83339 & 1757  &  83341 & 1748  &  83357 & 1742  \\
 83383 & 1765  &  83389 & 1757  &  83399 & 1755  &  83401 & 1747  &  83407 & 1745  &  83417 & 1758  &  83423 & 1746  &  83431 & 1758  \\
\hline
 \end{tabular}

 }\newpage

\normalsize{\textbf{Table 4.} Continue 19 }\medskip\\
\footnotesize{
 \renewcommand{\arraystretch}{0.96}
 \begin{tabular}
 {@{}c@{\,\,\,\,\,}c@{\,\,}|@{\,\,}c@{\,\,\,\,\,}c@{\,\,}|@{\,\,}c@{\,\,\,\,\,}c@{\,\,}|@{\,\,}c@{\,\,\,\,\,}c@{\,\,}|@{\,\,}c@{\,\,\,\,}c@{\,\,}|@{\,\,}c@{\,\,\,\,}c@{\,\,}|@{\,\,}c@{\,\,\,\,}c@{\,\,}|@{\,\,}c@{\,\,\,\,}c@{}}
  \hline
$q$ & $t_{2}^{L}$ &
$q$ & $t_{2}^{L}$ &
$q$ & $t_{2}^{L}$ &
$q$ & $t_{2}^{L}$ &
$q$ & $t_{2}^{L}$ &
$q$ & $t_{2}^{L}$ &
$q$ & $t_{2}^{L}$ &
$q$ & $t_{2}^{L^{\vphantom{H^{L}}}}$ \\
 \hline
 87253 & 1809  &  87257 & 1799  &  87277 & 1793  &  87281 & 1803  &  87293 & 1798  &  87299 & 1795  &  87313 & 1795  &  87317 & 1807  \\
 87323 & 1802  &  87337 & 1801  &  87359 & 1792  &  87383 & 1792  &  87403 & 1808  &  87407 & 1800  &  87421 & 1797  &  87427 & 1795  \\
 87433 & 1805  &  87443 & 1793  &  87473 & 1798  &  87481 & 1805  &  87491 & 1804  &  87509 & 1797  &  87511 & 1794  &  87517 & 1800  \\
 87523 & 1798  &  87539 & 1797  &  87541 & 1793  &  87547 & 1800  &  87553 & 1801  &  87557 & 1804  &  87559 & 1802  &  87583 & 1802  \\
 87587 & 1798  &  87589 & 1797  &  87613 & 1797  &  87623 & 1805  &  87629 & 1798  &  87631 & 1797  &  87641 & 1809  &  87643 & 1796  \\
 87649 & 1793  &  87671 & 1802  &  87679 & 1802  &  87683 & 1807  &  87691 & 1802  &  87697 & 1794  &  87701 & 1802  &  87719 & 1813  \\
 87721 & 1807  &  87739 & 1801  &  87743 & 1811  &  87751 & 1799  &  87767 & 1803  &  87793 & 1805  &  87797 & 1806  &  87803 & 1804  \\
\hline
 \end{tabular}

 }\newpage

\normalsize{\textbf{Table 4.} Continue 20 }\medskip\\
\footnotesize{
 \renewcommand{\arraystretch}{0.96}
 \begin{tabular}
 {@{}c@{\,\,\,\,\,}c@{\,\,}|@{\,\,}c@{\,\,\,\,\,}c@{\,\,}|@{\,\,}c@{\,\,\,\,\,}c@{\,\,}|@{\,\,}c@{\,\,\,\,\,}c@{\,\,}|@{\,\,}c@{\,\,\,\,}c@{\,\,}|@{\,\,}c@{\,\,\,\,}c@{\,\,}|@{\,\,}c@{\,\,\,\,}c@{\,\,}|@{\,\,}c@{\,\,\,\,}c@{}}
  \hline
$q$ & $t_{2}^{L}$ &
$q$ & $t_{2}^{L}$ &
$q$ & $t_{2}^{L}$ &
$q$ & $t_{2}^{L}$ &
$q$ & $t_{2}^{L}$ &
$q$ & $t_{2}^{L}$ &
$q$ & $t_{2}^{L}$ &
$q$ & $t_{2}^{L^{\vphantom{H^{L}}}}$ \\
 \hline
 91591 & 1847  &  91621 & 1845  &  91631 & 1845  &  91639 & 1840  &  91673 & 1843  &  91691 & 1850  &  91703 & 1847  &  91711 & 1843  \\
 91733 & 1849  &  91753 & 1849  &  91757 & 1836  &  91771 & 1856  &  91781 & 1845  &  91801 & 1852  &  91807 & 1852  &  91811 & 1856  \\
 91813 & 1852  &  91823 & 1855  &  91837 & 1850  &  91841 & 1845  &  91867 & 1841  &  91873 & 1847  &  91909 & 1841  &  91921 & 1844  \\
 91939 & 1842  &  91943 & 1849  &  91951 & 1848  &  91957 & 1854  &  91961 & 1851  &  91967 & 1850  &  91969 & 1852  &  91997 & 1857  \\
 92003 & 1842  &  92009 & 1854  &  92033 & 1844  &  92041 & 1844  &  92051 & 1845  &  92077 & 1854  &  92083 & 1853  &  92107 & 1852  \\
 92111 & 1841  &  92119 & 1850  &  92143 & 1860  &  92153 & 1841  &  92173 & 1850  &  92177 & 1853  &  92179 & 1849  &  92189 & 1846  \\
 92203 & 1856  &  92219 & 1847  &  92221 & 1854  &  92227 & 1845  &  92233 & 1856  &  92237 & 1852  &  92243 & 1856  &  92251 & 1849  \\
\hline
 \end{tabular}

 }\newpage

\normalsize{\textbf{Table 4.} Continue 21 }\medskip\\
\footnotesize{
 \renewcommand{\arraystretch}{0.96}
 \begin{tabular}
 {@{}c@{\,\,\,\,\,}c@{\,\,}|@{\,\,}c@{\,\,\,\,\,}c@{\,\,}|@{\,\,}c@{\,\,\,\,\,}c@{\,\,}|@{\,\,}c@{\,\,\,\,}c@{\,\,}|@{\,\,}c@{\,\,\,\,}c@{\,\,}|@{\,\,}c@{\,\,\,\,}c@{\,\,}|@{\,\,}c@{\,\,\,\,}c@{\,\,}|@{\,\,}c@{\,\,\,\,}c@{}}
  \hline
$q$ & $t_{2}^{L}$ &
$q$ & $t_{2}^{L}$ &
$q$ & $t_{2}^{L}$ &
$q$ & $t_{2}^{L}$ &
$q$ & $t_{2}^{L}$ &
$q$ & $t_{2}^{L}$ &
$q$ & $t_{2}^{L}$ &
$q$ & $t_{2}^{L^{\vphantom{H^{L}}}}$ \\
 \hline
 95869 & 1888  &  95873 & 1894  &  95881 & 1890  &  95891 & 1891  &  95911 & 1884  &  95917 & 1886  &  95923 & 1891  &  95929 & 1890  \\
 95947 & 1887  &  95957 & 1895  &  95959 & 1896  &  95971 & 1889  &  95987 & 1897  &  95989 & 1905  &  96001 & 1891  &  96013 & 1900  \\
 96017 & 1889  &  96043 & 1889  &  96053 & 1893  &  96059 & 1880  &  96079 & 1886  &  96097 & 1890  &  96137 & 1901  &  96149 & 1894  \\
 96157 & 1882  &  96167 & 1900  &  96179 & 1897  &  96181 & 1896  &  96199 & 1892  &  96211 & 1896  &  96221 & 1899  &  96223 & 1886  \\
 96233 & 1896  &  96259 & 1903  &  96263 & 1892  &  96269 & 1904  &  96281 & 1885  &  96289 & 1894  &  96293 & 1886  &  96323 & 1899  \\
 96329 & 1894  &  96331 & 1895  &  96337 & 1891  &  96353 & 1903  &  96377 & 1894  &  96401 & 1896  &  96419 & 1900  &  96431 & 1895  \\
 96443 & 1897  &  96451 & 1898  &  96457 & 1905  &  96461 & 1909  &  96469 & 1901  &  96479 & 1904  &  96487 & 1902  &  96493 & 1896  \\
\hline
 \end{tabular}
}

%% file: Table6_q_301831_321007.tex
\normalsize{ \textbf{Table 6.} The sizes
 $t_2^L=t_2^L(2,q)$
 of complete lexiarcs in planes
 $\mathrm{PG}(2,q),$\\ $301813< q\leq 321007$, $q$ prime power}\medskip\\
  \footnotesize{
 \renewcommand{\arraystretch}{0.92}
 \begin{tabular}
 {@{}c@{\,\,\,\,\,}c@{\,\,}|@{\,\,}c@{\,\,\,\,\,}c@{\,\,}|@{\,\,}c@{\,\,\,\,\,}c@{\,\,}|@{\,\,}c@{\,\,\,\,\,}c@{\,\,}|@{\,\,}c@{\,\,\,\,\,}c@{\,\,}|@{\,\,}c@{\,\,\,\,\,}c@{\,\,}|@{\,\,}c@{\,\,\,\,\,}c@{}}
  \hline
$q$ & $t_2^L$ &
$q$ & $t_2^L$ &
$q$ & $t_2^L$ &
$q$ & $t_2^L$ &
$q$ & $t_2^L$ &
$q$ & $t_2^L$ &
$q$ & $t_2^{L^{\vphantom{H^{L}}}}$ \\
 \hline
 301831 & 3519  &  301841 & 3526  &  301843 & 3522  &  301867 & 3527  &  301877 & 3526  &  301897 & 3506  &  301901 & 3522  \\
 301907 & 3526  &  301913 & 3523  &  301927 & 3523  &  301933 & 3522  &  301943 & 3536  &  301949 & 3522  &  301979 & 3516  \\
 301991 & 3526  &  301993 & 3527  &  301997 & 3521  &  301999 & 3528  &  302009 & 3510  &  302053 & 3528  &  302111 & 3527  \\
 302123 & 3533  &  302143 & 3523  &  302167 & 3522  &  302171 & 3514  &  302173 & 3526  &  302189 & 3521  &  302191 & 3527  \\
 302213 & 3517  &  302221 & 3511  &  302227 & 3535  &  302261 & 3539  &  302273 & 3518  &  302279 & 3520  &  302287 & 3529  \\
 302297 & 3514  &  302299 & 3512  &  302317 & 3510  &  302329 & 3517  &  302399 & 3521  &  302411 & 3517  &  302417 & 3537  \\
 302429 & 3530  &  302443 & 3524  &  302459 & 3525  &  302483 & 3524  &  302507 & 3517  &  302513 & 3525  &  302551 & 3513  \\
\hline
 \end{tabular}
}
 \newpage
 
\normalsize{ \textbf{Table 6.} Continue 1 }\medskip\\
 \footnotesize{
 \renewcommand{\arraystretch}{0.92}
 \begin{tabular}
 {@{}c@{\,\,\,\,\,}c@{\,\,}|@{\,\,}c@{\,\,\,\,\,}c@{\,\,}|@{\,\,}c@{\,\,\,\,\,}c@{\,\,}|@{\,\,}c@{\,\,\,\,\,}c@{\,\,}|@{\,\,}c@{\,\,\,\,\,}c@{\,\,}|@{\,\,}c@{\,\,\,\,\,}c@{\,\,}|@{\,\,}c@{\,\,\,\,\,}c@{}}
  \hline
$q$ & $t_2^L$ &
$q$ & $t_2^L$ &
$q$ & $t_2^L$ &
$q$ & $t_2^L$ &
$q$ & $t_2^L$ &
$q$ & $t_2^L$ &
$q$ & $t_2^{L^{\vphantom{H^{L}}}}$ \\
 \hline
 305917 & 3545  &  305927 & 3548  &  305933 & 3535  &  305947 & 3552  &  305971 & 3537  &  305999 & 3543  &  306011 & 3542  \\
 306023 & 3543  &  306029 & 3551  &  306041 & 3548  &  306049 & 3553  &  306083 & 3554  &  306091 & 3556  &  306121 & 3543  \\
 306133 & 3562  &  306139 & 3543  &  306149 & 3546  &  306157 & 3552  &  306167 & 3543  &  306169 & 3536  &  306191 & 3540  \\
 306193 & 3551  &  306209 & 3543  &  306239 & 3534  &  306247 & 3551  &  306253 & 3547  &  306259 & 3549  &  306263 & 3561  \\
 306301 & 3541  &  306329 & 3553  &  306331 & 3548  &  306347 & 3541  &  306349 & 3548  &  306359 & 3548  &  306367 & 3552  \\
 306377 & 3563  &  306389 & 3560  &  306407 & 3550  &  306419 & 3553  &  306421 & 3551  &  306431 & 3554  &  306437 & 3555  \\
 306457 & 3550  &  306463 & 3550  &  306473 & 3549  &  306479 & 3541  &  306491 & 3552  &  306503 & 3553  &  306511 & 3542  \\
\hline
 \end{tabular}
}
 \newpage

\normalsize{ \textbf{Table 6.} Continue 2 }\medskip\\
 \footnotesize{
 \renewcommand{\arraystretch}{0.92}
 \begin{tabular}
 {@{}c@{\,\,\,\,\,}c@{\,\,}|@{\,\,}c@{\,\,\,\,\,}c@{\,\,}|@{\,\,}c@{\,\,\,\,\,}c@{\,\,}|@{\,\,}c@{\,\,\,\,\,}c@{\,\,}|@{\,\,}c@{\,\,\,\,\,}c@{\,\,}|@{\,\,}c@{\,\,\,\,\,}c@{\,\,}|@{\,\,}c@{\,\,\,\,\,}c@{}}
  \hline
$q$ & $t_2^L$ &
$q$ & $t_2^L$ &
$q$ & $t_2^L$ &
$q$ & $t_2^L$ &
$q$ & $t_2^L$ &
$q$ & $t_2^L$ &
$q$ & $t_2^{L^{\vphantom{H^{L}}}}$ \\
 \hline
 310243 & 3574  &  310249 & 3570  &  310273 & 3585  &  310283 & 3575  &  310291 & 3584  &  310313 & 3576  &  310333 & 3568  \\
 310357 & 3566  &  310361 & 3565  &  310363 & 3560  &  310379 & 3568  &  310397 & 3582  &  310423 & 3580  &  310433 & 3570  \\
 310439 & 3565  &  310447 & 3569  &  310459 & 3561  &  310463 & 3576  &  310481 & 3579  &  310489 & 3571  &  310501 & 3564  \\
 310507 & 3573  &  310511 & 3560  &  310547 & 3574  &  310553 & 3568  &  310559 & 3574  &  310567 & 3573  &  310571 & 3575  \\
 310577 & 3570  &  310591 & 3571  &  310627 & 3578  &  310643 & 3576  &  310663 & 3563  &  310693 & 3578  &  310697 & 3571  \\
 310711 & 3580  &  310721 & 3572  &  310727 & 3577  &  310729 & 3573  &  310733 & 3571  &  310741 & 3570  &  310747 & 3578  \\
 310771 & 3571  &  310781 & 3578  &  310789 & 3579  &  310801 & 3579  &  310819 & 3580  &  310823 & 3572  &  310829 & 3574  \\
\hline
 \end{tabular}
}
 \newpage

\normalsize{ \textbf{Table 6.} Continue 3 }\medskip\\
 \footnotesize{
 \renewcommand{\arraystretch}{0.92}
 \begin{tabular}
 {@{}c@{\,\,\,\,\,}c@{\,\,}|@{\,\,}c@{\,\,\,\,\,}c@{\,\,}|@{\,\,}c@{\,\,\,\,\,}c@{\,\,}|@{\,\,}c@{\,\,\,\,\,}c@{\,\,}|@{\,\,}c@{\,\,\,\,\,}c@{\,\,}|@{\,\,}c@{\,\,\,\,\,}c@{\,\,}|@{\,\,}c@{\,\,\,\,\,}c@{}}
  \hline
$q$ & $t_2^L$ &
$q$ & $t_2^L$ &
$q$ & $t_2^L$ &
$q$ & $t_2^L$ &
$q$ & $t_2^L$ &
$q$ & $t_2^L$ &
$q$ & $t_2^{L^{\vphantom{H^{L}}}}$ \\
 \hline
 314441 & 3609  &  314453 & 3619  &  314467 & 3601  &  314491 & 3596  &  314497 & 3605  &  314513 & 3610  &  314527 & 3603  \\
 314543 & 3601  &  314549 & 3597  &  314569 & 3609  &  314581 & 3597  &  314591 & 3610  &  314597 & 3607  &  314599 & 3596  \\
 314603 & 3616  &  314623 & 3609  &  314627 & 3596  &  314641 & 3591  &  314651 & 3608  &  314693 & 3604  &  314707 & 3597  \\
 314711 & 3602  &  314719 & 3611  &  314723 & 3608  &  314747 & 3604  &  314761 & 3601  &  314771 & 3601  &  314777 & 3594  \\
 314779 & 3600  &  314807 & 3594  &  314813 & 3605  &  314827 & 3607  &  314851 & 3603  &  314879 & 3603  &  314903 & 3605  \\
 314917 & 3609  &  314927 & 3611  &  314933 & 3607  &  314953 & 3616  &  314957 & 3602  &  314983 & 3597  &  314989 & 3602  \\
 315011 & 3604  &  315013 & 3600  &  315037 & 3592  &  315047 & 3603  &  315059 & 3615  &  315067 & 3611  &  315083 & 3613  \\
\hline
 \end{tabular}
}
 \newpage

\normalsize{ \textbf{Table 6.} Continue 4 }\medskip\\
 \footnotesize{
 \renewcommand{\arraystretch}{0.92}
 \begin{tabular}
 {@{}c@{\,\,\,\,\,}c@{\,\,}|@{\,\,}c@{\,\,\,\,\,}c@{\,\,}|@{\,\,}c@{\,\,\,\,\,}c@{\,\,}|@{\,\,}c@{\,\,\,\,\,}c@{\,\,}|@{\,\,}c@{\,\,\,\,\,}c@{\,\,}|@{\,\,}c@{\,\,\,\,\,}c@{\,\,}|@{\,\,}c@{\,\,\,\,\,}c@{}}
  \hline
$q$ & $t_2^L$ &
$q$ & $t_2^L$ &
$q$ & $t_2^L$ &
$q$ & $t_2^L$ &
$q$ & $t_2^L$ &
$q$ & $t_2^L$ &
$q$ & $t_2^{L^{\vphantom{H^{L}}}}$ \\
 \hline
 318601 & 3631  &  318629 & 3625  &  318641 & 3629  &  318653 & 3632  &  318671 & 3636  &  318677 & 3636  &  318679 & 3617  \\
 318683 & 3639  &  318691 & 3636  &  318701 & 3615  &  318713 & 3623  &  318737 & 3636  &  318743 & 3630  &  318749 & 3630  \\
 318751 & 3623  &  318781 & 3618  &  318793 & 3627  &  318809 & 3625  &  318811 & 3623  &  318817 & 3625  &  318823 & 3615  \\
 318833 & 3628  &  318841 & 3636  &  318863 & 3632  &  318881 & 3636  &  318883 & 3637  &  318889 & 3629  &  318907 & 3629  \\
 318911 & 3629  &  318917 & 3630  &  318919 & 3635  &  318949 & 3636  &  318979 & 3629  &  319001 & 3633  &  319027 & 3629  \\
 319031 & 3635  &  319037 & 3633  &  319049 & 3634  &  319057 & 3645  &  319061 & 3641  &  319069 & 3623  &  319093 & 3621  \\
 319097 & 3635  &  319117 & 3625  &  319127 & 3637  &  319129 & 3624  &  319133 & 3632  &  319147 & 3621  &  319159 & 3629  \\
\hline
 \end{tabular}